\documentclass{amsart}
\usepackage{preamble}

\title{Area-Preserving Anisotropic Mean Curvature Flow in Two Dimensions}
\author{Eric Kim}
\address{Department of Mathematics, UCLA, Los Angeles, USA}
\email{ericykim@math.ucla.edu}

\author{Dohyun Kwon}
\address{Department of Mathematics, University of Seoul, Seoul, Republic of Korea}
\email{dh.dohyun.kwon@gmail.com}

\begin{document}

\begin{abstract}
    We study the motion of sets by anisotropic curvature under a volume constraint in the plane. We establish the exponential convergence of the area-preserving anisotropic flat flow to a disjoint union of Wulff shapes of equal area, the critical point of the anisotropic perimeter functional. This is an anisotropic analogue of the results in the isotropic case studied in \cite{julin2022}. The novelty of our approach is in using the Cahn-Hoffman map to parametrize boundary components as small perturbations of the Wulff shape. In addition, we show that certain reflection comparison symmetries are preserved by the flat flow, which lets us obtain uniform bounds on the distance between the convergent profile and the initial data. 
\end{abstract}

\maketitle

\newcommand{\phic}{\phi^\circ}
\newcommand{\psic}{\psi^\circ}
\newcommand{\Eh}{E^{(h)}}
\newcommand{\Ehn}{E^{(h_n)}}
\newcommand{\vh}{v^{(h)}}
\newcommand{\lambdah}{\lambda^{(h)}}
\newcommand{\sd}{\operatorname{sd}}

\section{Introduction}

In this paper, we study the long-term behavior of the flat flow solution to the area-preserving anisotropic mean curvature flow in the plane. The anisotropic mean curvature flow of sets $E_t\sb\R^N$, which preserves the volume $|E_t|$, is given by 
\begin{equation}
    \label{eq:volume aflow}
    V_t = \psic(\nu_{E_t})(-\kappa^\phi_{E_t} + \lambda_t)\quad \text{on }\partial E_t.
\end{equation}
Here, $V_t$ is the outward normal velocity along $\partial E_t$, $\nu_{E_t}$ is the outward normal vector, and $\phi,\psic:\mathbb{R}^N \rightarrow [0, \infty)$ are norms which represent surface energy density and mobility, respectively. Moreover,
\[ \lambda_t := \frac{\int_{\partial E_t} \kappa^\phi_{E_t}\psic(\nu_{E_t})d\cH^{N-1}}{\int_{\partial E_t} \psic(\nu_{E_t})d\cH^{N-1}} \]
is the Lagrange multiplier enforcing the volume constraint. The anisotropic mean curvature $\kappa^\phi_{E_t}$ 
represents the first variation of the anisotropic perimeter functional $P_\phi$, 
\begin{align}
\label{eq:aper}
    P_\phi(E) := \int_{\partial E} \phi(\nu_E) d\cH^{N-1} \hbox{ for a set } E\sb\R^N.
\end{align}
The volume-preserving flow \eqref{eq:volume aflow} arises in a number of applications, including in physics and material science to model solidification processes \cite{Carter1995, Tarshis1972}, as well as in image segmentation for computer vision \cite{imageprocessing}. Anisotropic surface energies arise naturally in the study of nematic liquid crystals, where the shared orientation of rod-shaped particles contributes to an anisotropic surface tension \cite{palmer2017}. 

The flow \eqref{eq:volume aflow} can be interpreted as a (formal) $L^2$-gradient flow of the functional $P_\phi$. The minimizer of $P_\phi$ among all sets of finite perimeter with a prescribed volume is a scaled and translated version of the Wulff shape
\begin{align*}
    W_\phi : = \set{x\in\R^N: x\cdot\nu \leq \phi(\nu)\ \forall \nu\in\R^N}.
\end{align*}
It has been shown in \cite{de2020uniqueness} that the only critical points of the volume-constrained problem are finite disjoint unions of Wulff shapes with equal area. Consequently, it is reasonable to expect the long-term convergence of the anisotropic mean curvature flow to these critical points under a volume constraint.

The main goal of our paper is to better understand the long-term behavior of \eqref{eq:volume aflow} in the plane. To this end, we present two types of results: (1) the exponential convergence of a flat flow solution of \eqref{eq:volume aflow} to a disjoint union of Wulff shapes given \textit{any} initial data, and (2) more precise characterizations of the limiting profile given additional assumptions on the initial data. For the latter, we are particularly interested in assumptions which allow for non-convexity of the initial data. The \textit{flat flow} solution that we consider is a minimizing movement scheme based on the gradient flow structure of \eqref{eq:volume aflow} and is discussed in further detail in \cref{sec:flat flow intro}. 

\subsection{Long-time behavior of the evolution}
Most results in the literature regarding the long-term convergence of \eqref{eq:volume aflow} have relied on some geometric property of the initial set which is preserved over time, such as convexity \cite{andrews2001volume, bellettini2009volume}, or a geometric condition associated with reflection symmetries of the Wulff shape \cite{kim2021volume}. Recently, \cite{julin2022} showed in the isotropic case that for \textit{any} initial set of finite perimeter in dimension $N=2$, the area-preserving flat flow converges exponentially to a disjoint union of equally sized disks. 

For our first main theorem, we generalize the results of \cite{julin2022} to the anisotropic setting. More specifically, we show that the \textit{area-preserving flat $(\phi,\psi)$-flow} in $N=2$ converges exponentially to a disjoint union of Wulff shapes of equal area. To the best of our knowledge, this is the first convergence result for \eqref{eq:volume aflow} in the anisotropic regime, which makes no geometric assumptions on the initial data. 

We let $\cM(\R^N)$ denote the space of norms on $\R^N$. For $\aa\in(0,1]$, we will let $\cM^2(\R^N)$ (resp. $\cM^{2,\aa}(\R^N)$) denote the space of all \textit{regular elliptic integrands} on $\R^N$ (see \cref{def:reg elliptic integrand}) which belong to $C^2(\S^{N-1})$ (resp. $C^{2,\aa}(\S^{N-1}))$. In most cases, we will consider a surface energy $\phi \in \cM^{2,1}(\R^N)$ and mobility $\psi\in \cM(\R^N)$. Interestingly, the regularity of $\psi$ (or lack thereof) does not play a role in obtaining convergence. In the following, $W_\phi(x,r)$ refers to the scaled and translated Wulff shape $x+rW_\phi$. 
We now present our first main result, whose proof is given in \cref{sec:lon}. 

\begin{theorem}  
    \label{thm:exponential convergence of flat flow}   
    For $\phi\in\cM^{2,1}(\R^2), \psi\in\cM(\R^2)$, let $\set{E(t)}_{t\geq0}$ be an area-preserving flat $(\phi,\psi)$-flow (defined in \cref{existence of flat flow}) starting from a bounded set of finite perimeter $E_0\sb\R^2$. Then there exists a disjoint union of Wulff shapes $E_\infty = \cup_{j=1}^d W_\phi(x_j, r)$ such that $|E_0| = |E_\infty|$, $P_\phi(E_\infty)\leq P_\phi(E_0)$, and such that for some constants $C,C_0>0$ and all $t\geq0$,
    \begin{equation} 
        \label{eq:exp conv of flat flow}
        \sup_{E(t)\Delta E_\infty}d^\psi_{E_\infty} \leq Ce^{-t/C_0}
    \end{equation}
    where $C_0$ depends only on 
    $\phi$, $E_0$, and $L_\psi$ given in \eqref{eq:lphi}. 
\end{theorem} 

We remark that in contrast to the analogous result in \cite{julin2022}, we are unable to show convergence of the energies $P_\phi(E(t))$ to $P_\phi(E_\infty)$. The difficulty arises from the absence of an almost-minimality property for disjoint unions of Wulff shapes; such a property relies on a calibration argument that does not seem to extend to the anisotropic setting easily. 

The primary ingredient needed to prove \cref{thm:exponential convergence of flat flow} is the following geometric result establishing that if a bounded set of finite perimeter has anisotropic mean curvature close to a constant (in the $L^2$ sense), then it is in fact well approximated by a disjoint union of Wulff shapes.

\begin{theorem}[Quantitative Alexandrov Theorem]
\label{QAT main}
For $\phi\in\cM^{2,1}(\R^2)$ and positive constants $m$ and $M$, there exist constants $\eps_0(\phi,m,M)\in(0,1)$ and $C(\phi, m,M)>0$ such that for any bounded $C^2$ open set $E\sb\R^2$ satisfying $|E|=m$, $P_\phi(E)\leq M$, and $\|\kappa_E^\phi - \ol{\kappa}_E^\phi\|_{L^2(\partial E)}\leq \eps_0$, the following hold:
\begin{enumerate}
    \item $E$ is diffeomorphic to a disjoint union of Wulff shapes $F = \cup_{j=1}^d W_\phi(x_j, r)$ such that $|E|=|F|$ and
    \begin{equation}
        \label{quadratic control of perimeter}
        \abs{P_\phi(E) - P_\phi(F)} \leq C\|\kappa_E^\phi - \ol{\kappa}_E^\phi\|_{L^2(\partial E)}^2.
    \end{equation}
    \item Each boundary component of $E$ may be parametrized as the normal graph over some $W_\phi(x_j, r)$, whose $C^{1,1/2}$ norm is less than $C \|\kappa_E^\phi - \ol{\kappa}_E^\phi\|_{L^2(\partial E)}$.
\end{enumerate}
\end{theorem}

The isotropic case of \cref{QAT main} was proven in \cite{julin2022} by using the Gauss-Bonnet theorem in the plane to obtain sufficient compactness, and then by showing that the arclength parametrization of $\partial E$ is well-approximated by circular arcs. The novelty of our approach is in exploiting the \textit{Cahn-Hoffman map}, which parametrizes the Wulff shape, to obtain suitable anisotropic analogues of both the Gauss-Bonnet theorem (\cref{anisotropic gauss bonnet}) and arclength parametrization. Some arguments are parallel to those in \cite{julin2022}, but it is delicate to verify that the quadratic exponent in \eqref{quadratic control of perimeter} is not lost in the anisotropic regime; see the discussion prior to \cref{pert into normal pert}.

The quadratic exponent on the righthand side of \eqref{quadratic control of perimeter} is sharp. Moreover, one can interpret \eqref{quadratic control of perimeter} as an infinite-dimensional example of the \textit{{\L}ojasiewicz inequality}, which states that near any critical point $z\in\R^N$ of an analytic function $f$, one has the estimate \begin{equation}
    \label{eq:lojasiewicz ineq}
    |f(x) - f(z)| \lesssim |\nabla f(x)|^\aa
\end{equation}
for some $\aa\in(0,2]$. Via a standard energy dissipation argument, one may use \eqref{eq:lojasiewicz ineq} to estimate the rate of convergence of the gradient flow of $f$, which is an exponential rate when $\aa=2$. Thus the quadratic exponent in \eqref{quadratic control of perimeter} is essential to obtaining the exponential decay in \cref{thm:exponential convergence of flat flow}, for which we apply a discretized dissipation argument.

An interesting open question is whether exponential convergence of \eqref{eq:volume aflow} is true for $N=3$; our method does not extend since the Gauss-Bonnet theorem (for mean curvature) holds only in the plane. In the isotropic case, \cite{julin2020} were able to establish convergence up to translations (but without a rate) by proving a weaker quantitative Alexandrov theorem. The authors have found that even the preliminary compactness result in \cite{julin2020} is not easily available to the anisotropic setting, as it relies on the Michael-Simon inequality, whose anisotropic analogue remains an open problem.

\subsection{Reflection Property} 
Once we have unconditional convergence of the flat flow from \cref{thm:exponential convergence of flat flow}, a natural follow-up question is whether we can constrain the arrangement of Wulff shapes in the convergent profile $E_\infty$, given additional assumptions on the initial data $E_0$. For instance, one might hope to bound the distance of $E_\infty$ from $E_0$. This does not follow immediately from \cref{thm:exponential convergence of flat flow}, since the rate of convergence depends delicately on the flat flow, which may not be unique. Interestingly, we can make progress by showing certain reflection comparison symmetries to be preserved by the flat flow if they are satisfied by the initial set.

First, we establish some notation and terminology. Given a half-space $H\sb\R^N$, we say that a set $E\sb\R^N$ satisfies the property $(*)_H$ (see \cref{fig:reflection}) if 
\begin{equation}
    \label{eq:reflection prop}
    \Psi(E) \cap H \sbq E\cap H\qquad \hbox{where $\Psi$ denotes reflection across $\partial H$}.
\end{equation}
Moreover, $E$ satisfies the stricter property $(*)_H'$ if 
\begin{equation}
    \label{eq:reflection prop'}
    \hbox{$E$ satisfies $(*)_H$}\qquad\text{and}\qquad \partial \Psi(E) \cap \partial E\sb \partial H.
\end{equation}
We say that a norm $\phi$ is \textit{compatible} with $\cP\sb\S^{N-1}$ if $\phi(x) = \phi(x - 2(x\cdot\nu)\nu)$ for all $\nu\in\cP, x\in \R^N$. Lastly, given subsets $E,D\sb\R^N$ and a set $\cP\sb\S^{N-1}$, we define $E$ to satisfy $(**)_{D,\cP}$ (resp. $(**)_{D,\cP}'$) if $E$ satisfies $(*)_H$ (resp. $(*)_H'$) for every half-space $H\sb\R^N$ which contains $D$ and is normal to a vector in $\cP$.

Our main result regarding the preservation of such properties is the following: 

\begin{theorem}
    \label{thm:reflection pres by flat flow}
    Let $N \in \{2,3\}$, $\phi\in\cM^{2,1}(\R^N)$, and $\psi\in\cM(\R^N)$ strictly convex. Consider a half-space $H$ with normal vector $\nu$ and suppose $\phi,\psi$ are compatible with $\nu$. If $E_0\sb\R^N$ is a bounded set of finite perimeter which satisfies $(*)_H'$ and is $C^1$ near $\partial H$, then any flat $(\phi,\psi)$-flow $E(t)$ with initial set $E_0$ satisfies $(*)_H$ for all $t\geq0$. 
\end{theorem}

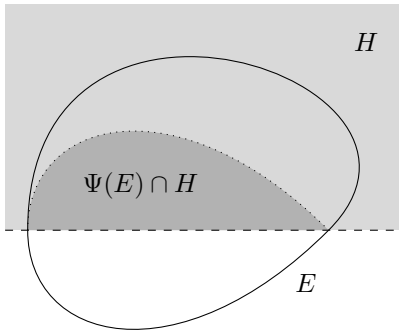
\begin{figure}
    \centering
    \begin{tikzpicture}
        \fill[gray!30!white] (5,0) -- (-0.3,0) -- (-0.3,3) -- (5,3) -- cycle; 
        \fill[gray!60!white] (0,0) .. controls (0,1.5) and (2,2) .. (4,0) -- cycle;
        
        \draw (0,0) .. controls (0,-1.5) and (2,-2) .. (4,0); 
        \draw (4,0) .. controls (6,2) and (0,4) .. (0,0); 
        \draw[dashed] (-0.3,0) -- (5,0);
        \draw[dotted] (0,0) .. controls (0,1.5) and (2,2) .. (4,0); 
        
        \draw (1.5,0.6) node{$\Psi(E) \cap H$};
        \draw (3.7,-0.7) node{$E$};
        \draw (4.5,2.5) node{$H$};
    \end{tikzpicture}
    \caption{Example of a set $E$ satisfying $(*)_H$}
    \label{fig:reflection}
\end{figure}

The property $(*)_H$ is motivated by the stronger notion of $\rho$-reflection first studied by \cite{Feldman2014}. The preservation of $(*)_H$ has been established for continuous viscosity solutions of \eqref{eq:volume aflow} in the isotropic case \cite{kim2020mean, Kim2020star} and the anisotropic case \cite{kim2021volume}. Our results extend this perspective by investigating the preservation of $(*)_H'$ along the flat flow, which is new, even in the isotropic case. Without access to the comparison principle, our proof of \cref{thm:reflection pres by flat flow} proceeds by constructing an energy competitor in the event that the property $(*)_H'$ is violated. The difficulty in our case is that the energy competitor needs to be a strict improvement due to the nonuniqueness of the approximate flow, and this is why we assume the stronger property $(*)_H'$ for the initial data.

For $N=2$, one may apply \cref{thm:exponential convergence of flat flow} to constrain the possible arrangements of Wulff shapes in the convergent profile in various ways. For instance, in \cref{cor:diameter bound} below, we establish a bound on $E_\infty$ which depends only on the diameter of $E_0$; see \cref{fig:bounded}. In the absence of a comparison principle, the standard barrier argument is not applicable to \eqref{eq:volume aflow}. Our finding, although straightforward, serves as the only uniform bound for the limit of \eqref{eq:volume aflow}.

\begin{corollary}
    \label{cor:diameter bound}
    Let $\phi\in \cM^{2,1}(\R^2), \psi\in\cM(\R^2)$ be compatible with some set $\cP\sb\S^1$, such that $\psi$ is strictly convex and $\text{span}(\cP)=\R^2$. Then, the limiting set $E_\infty$ of the volume-preserving flat $(\phi,\psi)$-flow starting from a bounded set $E_0$ of finite perimeter is uniformly bounded:
    \begin{align}
        E_\infty \sbq D + (|E_0|/|W_\phi|)^{1/2}W_\phi
    \end{align}
    for any convex polygon $D\sb\R^2$ whose edges are normal to vectors in $\cP$ and $E \subset D$.
\end{corollary}

As another application, in \cref{cor:single wulff} we show that sufficiently many reflection comparison symmetries force $E_0$ to converge to a single Wulff shape; see \cref{fig:single wulff}.

\begin{corollary}
    \label{cor:single wulff}
    Let $\phi\in \cM^{2,1}(\R^2), \psi\in\cM(\R^2)$ be compatible with some set $\cP\sb\S^1$, such that $\psi$ is strictly convex. If $E_0\sb\R^2$ is a $C^1$ set of finite perimeter and satisfies $(**)_{D,\cP}'$ for a convex polygon $D\sb\R^2$ whose edges are normal to vectors in $\cP$ and such that \[ |E_0| > 2(\sqrt{\aa^2+1}+\aa)^2|D| \qquad \text{where}\qquad \aa := \frac{P_\phi(D)}{2|W_\phi|^{1/2}|D|^{1/2}}, \]
        then every area-preserving flat $(\phi,\psi)$-flow $E(t)$ starting from $E_0$ converges to a single Wulff shape. 
\end{corollary}

Ideally, one would like to weaken the hypothesis of \cref{thm:reflection pres by flat flow} so that the initial set $E_0$ does not need to be $C^1$ near $\partial H$ or need only satisfy the non-strict property $(*)_H$ as opposed to $(*)_H'$. To this end, we show a form of stability of the volume-preserving flat $(\phi,\psi)$-flow with respect to initial data. While stability does not hold true in general, it is true if we enforce enough reflection symmetries on the initial data to obtain compactness via a uniform cone condition (\cref{prop:reflection lipschitz}). Moreover, due to the absence of uniqueness, we are only able to prove an existential form of symmetry preservation:

\begin{theorem}
    \label{thm:reflection pres by flat flow stability}
    For $N \in \{2,3\}$, let $\phi\in \cM^{2,1}(\R^2), \psi\in\cM(\R^2)$ be compatible with some set $\cP\sb\S^1$, such that $\psi$ is strictly convex. Suppose $\cP\sb\S^{N-1}$ is a root system [see \eqref{eq:root system}] such that $\text{span}(\cP\setminus K) = \R^N$ for any hyperplane $K$ through the origin. Then there exists $c = c(\cP)$ such that the following holds: For any bounded set of finite perimeter $E_0\sb\R^N$ satisfying $(**)_{B_\rho(0),\cP}$ for $\rho< c|E_0|^{1/N}$, there exists a flat $(\phi,\psi)$-flow $E(t)$ starting at $E_0$ which satisfies $(**)_{B_\rho(0),\cP}$ for all $t\geq0$. 
\end{theorem}

\begin{figure}
    \centering
    \begin{tikzpicture}
        \fill[gray!30!white] (0.5, 1.866) -- 
        (2, 1.866) arc(90:30:1) -- 
        (3.583,0.067) arc(30:-30:1) --
        (3.366, -1.366) arc(330:270:1) --
        (0.5, -1.866) arc(270:210:1) --
        (-0.866, -0.5) arc(210:150:1) --
        (-0.366, 1.366) arc(150:90:1);

        \draw (0.5, 1.866) -- 
        (2, 1.866) arc(90:30:1) -- 
        (3.583,0.067) arc(30:-30:1) --
        (3.366, -1.366) arc(330:270:1) --
        (0.5, -1.866) arc(270:210:1) --
        (-0.866, -0.5) arc(210:150:1) --
        (-0.366, 1.366) arc(150:90:1);
    
        \fill[gray!75!white] 
        (0.25, 0.433) .. controls (1,0.433) and (1, 0.866) .. 
        (1.5, 0.866) .. controls (3,0.866) and (1,0) .. 
        (2.75,-0.433) .. controls (2,-0.8) and (1.5,-0.866) .. 
        (1,-0.866) .. controls (0.5,-0.866) and (0.3,-0.52) .. 
        (0.25, -0.433) .. controls (0,0) and (0.125,.217) .. cycle;
    
        \draw[dashed] (0,0) -- (0.5,0.866) -- (2, 0.866) -- (2.75,-0.433) -- (2.5, -0.866) -- (0.5, -0.866) -- cycle;
        
        \draw (1.1,0) node{$E_0$};
        \draw (0.8, 1.1) node{$D$};
        \draw (1.25, 2.1) node{$D + r W_\phi$};
    
        \draw (3.6,1.5) node{$\cP$};
        \def\x{4.5};
        \def\y{1};
        \def\r{0.75};
    
        \foreach \n in {0,...,5}
            \draw[->] (\x,\y) -- ({\x + \r*cos(30+60*\n)},{\y + \r*sin(30+60*\n)});
    \end{tikzpicture}
    \caption{Example of \cref{cor:diameter bound}}
    \label{fig:bounded}
\end{figure}

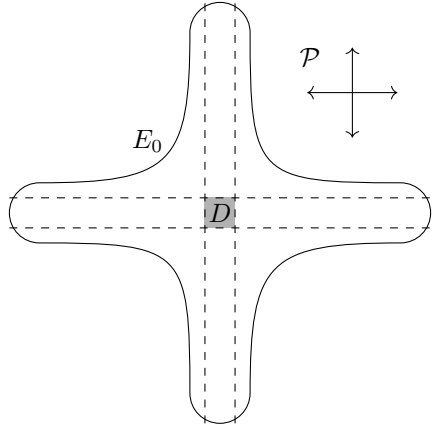
\begin{figure}
    \centering
    \begin{tikzpicture}[scale=0.8]
        \fill[gray!60!white] (-0.25,-0.25) rectangle (0.25,0.25);
        
        \def\e{2.2};
        \def\f{2.4};
        \draw (-3,-0.5) arc(270:90:0.5) .. controls (-3+\e,0.5) and (-0.5,3-\e) .. 
        (-0.5,3) arc(180:0:0.5) .. controls (0.5,3-\f) and (3-\f,0.5) .. 
        (3,0.5) arc(90:-90:0.5) .. controls (3-\e,-0.5) and (0.5,-3+\e) ..
        (0.5,-3) arc(0:-180:0.5) .. controls (-0.5,-3+\f) and (-3+\f,-0.5) .. cycle;
    
        \draw[dashed] (-3.5,-0.25) -- (3.5,-0.25);
        \draw[dashed] (-3.5,0.25) -- (3.5,0.25);
        \draw[dashed] (-0.25,-3.5) -- (-0.25,3.5);
        \draw[dashed] (0.25,-3.5) -- (0.25,3.5);
    
        \draw (0,0) node{$D$};
        \draw (-1.2,1.2) node{$E_0$};
    
        \draw (1.5,2.6) node{$\cP$};
        \def\x{2.2};
        \def\y{2};
        \def\r{0.75};
    
        \foreach \n in {0,...,3}
            \draw[->] (\x,\y) -- ({\x + \r*cos(90*\n)},{\y + \r*sin(90*\n)});
    \end{tikzpicture}
    \caption{Example of initial data $E_0$ converging to a single Wulff shape by \cref{cor:single wulff}}
    \label{fig:single wulff}
\end{figure}

\subsection{The Flat Flow}
\label{sec:flat flow intro}
Let us now clarify our notion of solution. It is well-known that \eqref{eq:volume aflow} can encounter topological singularities in finite time, such as self-intersections and pinch-offs \cite{singularity}, so weaker formulations of the flow are needed to grant global-in-time existence. Although existence of solutions to \eqref{eq:volume aflow} is well studied in the isotropic case \cite{mugnai2016, Takasao2023}, existence for a general anisotropy remains an open question unless the initial set is convex \cite{andrews2001volume, bellettini2009volume} or satisfies a geometric property associated with reflection symmetries of the Wulff shape \cite{kim2021volume}.

In our paper, we define a notion of flat flow solution to \eqref{eq:volume aflow}, which is a natural choice to accommodate the gradient flow structure. Almgren, Taylor, and Wang \cite{Almgren1993}, and Luckhaus and Sturzenhecker \cite{luckhaus1995implicit} first introduced a flat flow solution to the unforced mean curvature flow using a minimizing movements approach, which takes a limit of discrete flows obtained from iterating an energy minimization problem. Later the flat flow solution was adapted by Mugnai, Seis, and Spadaro \cite{mugnai2016} to \eqref{eq:volume aflow} in the isotropic case, by incorporating a soft volume penalization term which becomes a hard constraint in the limiting flow. In this paper, we adapt the construction of \cite{mugnai2016} to define a flat $(\phi,\psi)$-flow solution to \eqref{eq:volume aflow}. It is worth noting that, in contrast to the unforced flow, existence in the volume-preserving case is considerably more delicate due to its nonlocal nature and the absence of a comparison principle. We also expect that under a standard energy convergence assumption, the flat $(\phi,\psi)$-flow is also a distributional solution to \eqref{eq:volume aflow}. However, this is not the focus of the present work, so we do not comment on this further.

\subsection{Outline of the Paper}
In Section 3, we prove \cref{QAT main} using the Cahn-Hoffman map. In Section 4, we establish the existence of volume-preserving flat $(\phi,\psi)$-flows and necessary estimates for the long-term behavior. \cref{thm:exponential convergence of flat flow} is proven in Section 5. In Section 6, we prove \cref{thm:reflection pres by flat flow} as well as some corollaries regarding the long-term profile. 

\subsection{Acknowledgements}
The authors thank Inwon Kim for helpful comments on the manuscript. DK was partially supported by the National Research Foundation of Korea (NRF) grant funded by the Korea government (MSIT) (No. RS-2023-00252516) and the POSCO Science Fellowship of POSCO TJ Park Foundation.

\section{Preliminaries}
\label{sec:prelim}

\subsection{Notations}
A Lebesgue measurable set $E\sb\R^N$ of finite perimeter is one such that the distributional gradient $\mu_E := -D1_E$ is a finite Borel measure, in which case the perimeter $P(E)$ is defined to be the total variation $|\mu_E|(\R^N)$. The reduced boundary $\partial^*E$ is defined as the set of points $x\in \spt \mu_E$ such that the following limit exists: \[ \nu_E(x) := \lim_{r\to0^+} \frac{\mu_E(B_r(x))}{|\mu_E|(B_r(x))}. \]
The structure theorem of De Giorgi establishes that $\mu_E = \nu_E \cH^{N-1}|_{\partial^*E}$ and $|\mu_E| = \cH^{N-1}|_{\partial^*E}$. We refer the reader to \cite{maggi2012} for further background on sets of finite perimeter.

\begin{definition}
    \label{def:reg elliptic integrand}
    We say that a norm $\phi:\R^N\to \R_{\geq0}$ is a \textbf{regular elliptic integrand} if the restriction $\gamma := \phi|_{\S^{N-1}}$ is $C^2$ and strictly positive, and there exists a constant $c>0$ such that \begin{equation}
        e\cdot(D^2\phi(\nu)e) \geq c|e - (\nu\cdot e)\nu|^2 \quad\forall \nu,e\in\S^{N-1}.
    \end{equation}
\end{definition}
In this case, there is an \textit{ellipticity constant} $\Lambda_\phi>0$ such that 
\begin{equation}
    \label{eq:ellipticity const}
    \Lambda_\phi^{-1} I_{N-1} \leq \gamma(\nu) I_{N-1} + D^2\gamma(\nu) \leq \Lambda_\phi I_{N-1}\qquad \forall \nu\in\S^{N-1}.
\end{equation}

Here is a list of notations we use throughout the paper:
\begin{itemize}
    \item 
    Given a norm $\psi$ on $\R^N$ and a set $E\sb\R^N$, we define the signed $\psi$-distance function \[ \sd^\psi_E(x) := \begin{cases}
    \dist^\psi(x, E) &x\in E^c\\
    -\dist^\psi(x, E^c) &x\in E
\end{cases} \] 
where $\dist^\psi(x,E) := \inf\set{\psi(x-y):y\in E}$. We also denote $d^\psi_E(x) := |\sd^\psi_E(x)| = \dist^\psi(x,\partial E)$. We will always identify a bounded set of finite perimeter $E\sb\R^N$ with its set of Lebesgue points so that the function $\sd^\psi_E$ is well-defined.
    \item 
    Given a norm $\phi$ on $\R^N$, we denote $L_\phi$ to be the smallest constant such that 
\begin{align}
\label{eq:lphi}
    L_\phi^{-1} \leq \phi(\nu) \leq L_\phi \qquad\forall\nu\in\S^{N-1}.
\end{align}
We denote the dual norm
\[
    \phi^\circ(y) := \sup\set{x\cdot y: \phi(x) \leq 1}
\] and recall that $(\phi^{\circ})^\circ = \phi$. 
It is standard to check that $L_{\phic} = L_\phi$.
\item
We say that a set $E\sb\R^N$ satisfies the \textit{\textbf{interior $rW_\phi$-property}} at $x\in\partial E$ if there is some $y\in E$ such that $W_\phi(y,r)\sb E$ and $x\in \partial W_\phi(y,r)$.
\end{itemize}

\subsection{Anisotropic perimeter and mean curvature}

For a norm $\phi$, we define the associated $\phi$-perimeter as \[ P_\phi(E;A) := \int_{\partial^*E\cap A} \phi(\nu_E) d\cH^{N-1} \]
for any set $E\sb\R^N$ of locally finite perimeter and any set $A\sb\R^N$. We sometimes refer to the anisotropic surface density $dP_\phi := \phi(\nu) d\cH^{N-1}$.

\begin{definition}
For $\phi\in\cM^2(\R^N)$, we say that a bounded set $E\sb\R^N$ of finite perimeter has (scalar) \textbf{mean $\phi$-curvature} $\kappa_E^\phi \in L^1(\partial^*E)$ if 
\begin{equation}
    \label{eq:curv def}
    \int_{\partial E} \kappa_E^\phi\nu\cdot X d\cH^{N-1} = \int_{\partial E} \div_{\tau,\phi} X dP_\phi \qquad \forall X\in C^1_c(\R^N;\R^N)
\end{equation}
where
\begin{align*}
    \div_{\tau,\phi} X := \operatorname{tr}\ps{\ps{I - \nabla\phi(\nu_E)\otimes \frac{\nu_E}{\phi(\nu_E)}}\nabla X}.
\end{align*}
\end{definition}
One can check that the righthand side of \eqref{eq:curv def} is the first variation of $P_\phi$ along $X$. In the case where $E$ is $C^2$, we have $\kappa_E^\phi = \div_\tau \nabla \phi(\nu_E)$. In particular, when $\partial E$ is a $C^2$ curve in $\R^2$, one has the simpler expression $\kappa_E^\phi = \kappa_E(\gamma+\gamma'')(\nu_E)$ where $\kappa_E$ is the isotropic mean curvature. It is a standard result that the Wulff shape has $\kappa^\phi\equiv 1$. 


For $\phi\in \cM^2(\R^N)$, $\partial W_\phi$ is explicitly parametrized by the Cahn-Hoffman map $\xi:\S^{N-1}\to\R^N$ \cite{koiso2021}
    \[ \xi(\nu) := \nabla\phi(\nu) = \gamma(\nu)\nu + \nabla\gamma(\nu) \]
where the gradient $\nabla\gamma(\nu)$ is naturally embedded into the tangent plane to $\nu$ in $\R^N$. Moreover, $\xi$ is a $C^1$-embedding and each $\nu\in\S^{N-1}$ is in fact the outward normal vector to $W_\phi$ at $\xi(\nu)$. The Cahn-Hoffman map will be a critical tool for parametrizing boundary components in the proof of \cref{QAT main}.

\section{Quantitative Alexandrov Theorem}
In this section, we establish one of our main results, \cref{QAT main}. Some of the arguments are parallel to the ones in the isotropic case presented in \cite{julin2022}. The first new ingredient is an anisotropic version of the Gauss-Bonnet theorem for curves shown below. We recall the notation $\gamma := \phi|_{\S^1}$.

\begin{lemma}[Anisotropic Gauss-Bonnet theorem]
    \label{anisotropic gauss bonnet}
    For $\phi\in \cM^2(\R^2)$ and any closed $C^2$ curve $\Gamma\sb\R^2$, \[ \int_\Gamma \kappa^\phi dP_\phi = 2|W_\phi|. \]
\end{lemma}
\begin{proof}
    We parametrize $\Gamma$ by arclength via $\zeta : [0,l]\to \R^2$, and let $\theta(s)$ be the angle of the normal vector to $\zeta(s)$, so that $\theta'(s) = \kappa(\zeta(s))$ and $\theta(0) = 0$. Identifying $\S^1\simeq[0,2\pi]$, we can reparametrize the integral
    \begin{align*}
        \int_\Gamma \kappa^\phi dP_\phi &= \int_\Gamma \kappa [\gamma(\nu)+\gamma''(\nu)]\gamma(\nu) d\cH^1 \\
        &= \int_0^l \theta'(s) [\gamma(\gamma + \gamma'')](\theta(s))ds\\
        &= \int_0^{2\pi} \gamma(\gamma + \gamma'') d\theta.
    \end{align*}
    The last expression is in fact equal to $P_\phi(W_\phi)$, by parametrizing with the Cahn-Hoffman map $\xi(\nu) = \gamma(\nu) \nu + \nabla\gamma(\nu)$. In terms of $\theta$ coordinates, one can compute \begin{equation}
        \label{speed of Cahn-Hoffman}
        |\xi'(\theta)| = (\gamma + \gamma'')(\theta).
    \end{equation} 
    and thus \begin{align*}
        \int_0^{2\pi} \gamma(\gamma + \gamma'') d\theta &= \int_0^{2\pi} |\xi'(\theta)|\gamma(\theta)d\theta = \int_{\partial W_\phi} \gamma(\nu) d\cH^1 = P_\phi(W_\phi).
    \end{align*}
    Since the Wulff shape solves the anisotropic isoperimetric inequality \eqref{isoperimetric inequality}, we have $P_\phi(W_\phi) = 2|W_\phi|$, concluding the proof.
\end{proof}

Using \cref{anisotropic gauss bonnet}, one is able to show the following compactness result in Proposition~\ref{QAT lemma}. The proof is a direct adaptation of the argument from \cite[Proposition 2.1]{julin2022}, which we provide in Appendix~\ref{ap:proof3} for the sake of completeness. 

\begin{proposition}
\label{QAT lemma}
    Let $M,m>0$ and $\phi\in \cM^2(\R^2)$. Then there exist constants $c,C>0$ and $\eps_0\in (0,1)$ which depend on $L_\phi$,$M,m$ such that the following holds: Let $E\sb\R^2$ be a bounded $C^2$ open set such that $|E| = m$, $P_\phi(E) \leq M$, and $\|\kappa_E^\phi - \ol{\kappa}_E^\phi\|_{L^2(\partial E)} \leq \eps_0$. Then \begin{itemize}
        \item[(a)] $c \leq \ol{\kappa}_E^\phi \leq C$
        \item[(b)] $E$ consists of at most $C$ components, which are all simply connected.
    \end{itemize} 
\end{proposition}

Next, provided $\phi\in\cM^{2,1}(\R^2)$, we explicitly parametrize each boundary component of $E$ as a $C^{1,1/2}$-small perturbation of a scaled Wulff shape. In the isotropic case, parametrizing the boundary $\partial E$ with respect to arclength is enough, as argued in \cite{julin2022}. In our case, we parametrize at a weighted speed associated to the anisotropy, and show that the angle along such a parametrization is approximately linear.

\begin{proposition}
\label{prop:pert}
Under the same setting as in \cref{QAT lemma}, we additionally assume that $\phi\in \cM^{2,1}(\R^2)$.
The boundary of each connected component $E_j$ of $E$ can be parametrized by $\xi_j + \sigma_j$ where for some $C = C(L_\phi, \Lambda_\phi,\|\phi\|_{C^{2,1}(\S^1)},m,M)>0$, $x_j \in \mathbb{R}^2$ and $r_j>0$,
\begin{enumerate}
    \item $\xi_j$ is the arclength parametrization of a Wulff shape $\partial W_\phi(x_j, r_j)$ satisfying $|E_j| = |W_\phi(x_j, r_j)|$ and $|r_j - (\ol{\kappa}_E^\phi)^{-1}| \leq C\|\kappa_E^\phi - \ol{\kappa}_E^\phi\|_{L^2(\partial E)}$, and 
    \item $\|\sigma_j\|_{C^{1,1/2}} \leq C \|\kappa_E^\phi - \ol{\kappa}_E^\phi\|_{L^2(\partial E)}$.
\end{enumerate}

\end{proposition}

\begin{proof}
    In what follows, $C$ is a constant depending only on $L_\phi, \Lambda_\phi,\|\phi\|_{C^{2,1}(\S^1)}$, $m$, and $M$ which may change from line to line. We recall that $\Lambda_\phi^{-1} \leq \gamma + \gamma'' \leq \Lambda_\phi$. For convenience, we denote $\eps := \|\kappa_E^\phi - \ol{\kappa}_E^\phi\|_{L^2(\partial E)}$.
    
    Fix a connected component $F\sb E$ with boundary $\Gamma$, let $\zeta:[0,\ell]\to\R^2$ be an arclength parametrization of $\Gamma$, and let $\theta(s)$ be the angle of the outward normal to $\Gamma$ at $\zeta(s)$. Without loss of generality, we set $\theta(0)=0$.

    To obtain initial estimates on $\theta$, it will be convenient to momentarily reparametrize to $\tilde{\zeta}(t) = \zeta(h(t))$ and $\tilde{\theta}(t) = \theta(h(t))$ (with $h(0)=0$) so that we have the speed
    \begin{equation}
        \label{reparametrized speed}
        |\tilde{\zeta}'(t)| = (\gamma + \gamma'')(\tilde{\theta}(t))
    \end{equation}
    and hence the relation
    \begin{equation}
        \label{angle curvature relation}
        \tilde{\theta}'(t) = \kappa_E(\tilde{\zeta}(t))(\gamma + \gamma'')(\tilde{\theta}(t)) = \kappa_E^\phi(\tilde{\zeta}(t)).
    \end{equation}
    Note \eqref{reparametrized speed} is equivalent to the ODE $h'(t) = (\gamma+\gamma'')(\theta(h(t))$, so $h$ exists and is unique in $C^2$. By \eqref{angle curvature relation} and the fact that $|\tilde{\zeta}'(t)|\leq \Lambda_\phi$,
    we find that $\tilde{\theta}$ grows almost linearly, with error controlled by $\eps$: for any $t\in [0,h^{-1}(\ell)]$, \begin{equation}
        \label{eq:angle almost linear}
        |\tilde{\theta}(t) - \ol{\kappa}_E^\phi t| \leq \int_0^t |\kappa_E^\phi(\tilde{\zeta}(s)) - \ol{\kappa}_E^\phi| ds \leq \Lambda_\phi^{1/2} t^{1/2} \|\kappa_E^\phi - \ol{\kappa}_E^\phi\|_{L^2(\partial E)} \leq C \eps.
    \end{equation}
    Let us denote the tangent and normal vectors $\tau(\theta) := (-\sin\theta, \cos\theta)$ and $n(\theta) := (\cos\theta,\sin\theta)$. By \eqref{eq:angle almost linear}, we have the estimate on velocities
    \begin{equation}
        \label{eq:derivative estimate for perturbation}
        |\tilde{\zeta}'(t) - \xi'(\ol{\kappa}_E^\phi t)| = |(\gamma + \gamma'')(\tilde{\theta}(t))\tau(\tilde{\theta}(t)) - (\gamma+\gamma'')(\ol{\kappa}_E^\phi t)\tau(\ol{\kappa}_E^\phi t)| \leq \text{Lip}(\gamma + \gamma'') |\tilde{\theta}(t) - \ol{\kappa}_E^\phi t| \leq C\eps.
    \end{equation}
    Thus by integrating the previous bound, we obtain that for some $x_0\in\R^2$ and all $t \in [0,h^{-1}(\ell)]$, \begin{equation}
        \label{eq:Linfty estimate for perturbation}
        |\tilde{\zeta}(t) - x_0 - (\ol{\kappa}_E^\phi)^{-1} \xi(\ol{\kappa}_E^\phi t)| \leq C\eps.
    \end{equation}
    Since $\phi$ is equivalent to the Euclidean norm, we have the containments \[ W_\phi(x_0,(\ol{\kappa}_E^\phi)^{-1} - C\eps) \sbq F \sbq W_\phi(x_0, (\ol{\kappa}_E^\phi)^{-1} + C\eps) \] so that by the intermediate value theorem, there exists $|r - (\ol{\kappa}_E^\phi)^{-1}| \leq C\eps$ such that $|F| = |W_\phi(x_0, r)|$. By repeating the calculations for \eqref{eq:angle almost linear}, \eqref{eq:derivative estimate for perturbation}, and \eqref{eq:Linfty estimate for perturbation}, one finds that $|\tilde{\theta}(t) - r^{-1}t| \leq C\eps$ for all $t\in [0,h^{-1}(\ell)]$ and hence
        \[ \|\tilde{\zeta}(t) - x_0 - r\xi(r^{-1}t)\|_{C^1} \leq C\eps. \]
    
    To obtain the sharper $C^{1,1/2}$-bound, we return to arclength parametrization. Denote the perturbation \[ \sigma(s) := \zeta(s) - x_0 - r \xi(r^{-1} h^{-1}(s)) \] for which we have shown $\|\sigma\|_{C^1}\leq C\eps$.
    Note that $r \geq 1/C$ for $\eps_0$ sufficiently small by \cref{QAT lemma}. Since both $\zeta$ and $\xi\circ h^{-1}$ travel at unit speed, we may bound for all $s\in[0,\ell]$
    \begin{align}
        \label{eq:C^1,1/2 estimate I}
        |\sigma''(s)| &= \abs{[\tau(\theta(s)) - \tau(r^{-1} h^{-1}(s))]'}\\
        &= \abs{\frac{\kappa_E^\phi(\zeta(s))}{(\gamma + \gamma'')(\theta(s))}n(\theta(s)) -\frac{r^{-1}}{(\gamma + \gamma'')(r^{-1} h^{-1}(s))}n(r^{-1}h^{-1}(s))}\notag \\
        &\leq \frac{|\kappa_E^\phi(\zeta(s)) - r^{-1}|}{(\gamma + \gamma'')(\theta(s))} + r^{-1} \abs{\frac{n(\theta(s))}{(\gamma+\gamma'')(\theta(s))} -\frac{n(r^{-1} h^{-1}(s))}{(\gamma+\gamma'')(r^{-1} h^{-1}(s))} } \notag\\
        &\leq \Lambda_\phi|\kappa_E^\phi(\zeta(s)) - r^{-1}| + \text{Lip}\ps{\frac{n(\cdot)}{(\gamma+\gamma'')(\cdot)}} |\theta(s) - r^{-1} h^{-1}(s)| \notag\\
        &\leq C\ps{|\kappa_E^\phi(\zeta(s)) - \ol{\kappa}_E^\phi| + \eps}. \notag
    \end{align}
    Integrating the previous bound then yields that for any $0\leq s < t \leq \ell$,
    \begin{align}
        \label{eq:C^1,1/2 estimate II}
        |\sigma'(s) - \sigma'(t)| &\leq C\ps{\int_s^t |\kappa_E^\phi(\zeta(s)) - \ol{\kappa}_E^\phi|ds + C\eps|s-t|}\\
        &\leq C\ps{\|\kappa_E^\phi - \ol{\kappa}_E^\phi\|_{L^2(\partial E)}|s-t|^{1/2} + \eps |s-t| } \leq C\eps |s-t|^{1/2}\notag
    \end{align}
    and hence $\|\sigma'\|_{C^{1/2}} \leq C\eps$, completing the proof.
\end{proof}

The perturbation given in \cref{prop:pert} can be recast as a normal perturbation of a Wulff shape. This will be used in conjunction with \cref{area and perimeter approx for normal perturbations} to approximate the area and anisotropic perimeter of $E$.

We also note that the argument in \cite[Proposition 2.1]{julin2022} made use of a quantitative Alexandrov theorem for normal perturbations of the sphere, taking the form \begin{equation}
    \label{eq:QAT for normal pert}
    \|f\|_{H^1(\S^1)} \leq C\|\kappa_{E_f} - \ol{\kappa}_{E_f}\|_{L^2(\partial E_f)} 
\end{equation}
where $E_f\sb\R^2$ is the region bounded by the graph $x\in \S^1 \mapsto (1+f(x))x$. There does not seem to exist an anisotropic analogue of \eqref{eq:QAT for normal pert}. However, we find that with sufficient bookkeeping, the use of such an inequality is not necessary to prove \cref{QAT main}. 

To be more precise, \eqref{eq:QAT for normal pert} was needed in \cite{julin2022} because they only asserted an estimate of the form \[ \|f\|_{C^{1,1/2}} \leq \omega(\|\kappa_{E_f}^\phi - \ol{\kappa}_{E_f}^\phi\|_{L^2}) \] for some modulus of continuity $\omega$. However, \cref{pert into normal pert} implies that $\omega$ is, in fact, linear.

\begin{lemma}
    \label{pert into normal pert}
    Let $\Gamma\sb\R^2$ be a $C^2$ curve with arclength parametrization $\zeta(s)$. Then there exists $\eps >0$ such that the following holds: any curve given by $\tilde{\zeta}(s) = \zeta(s) + \sigma(s)$ satisfying $\|\sigma\|_{C^1}<\eps$ can instead be parametrized as a normal perturbation of $\Gamma$, i.e. as $x + f(x) \nu(x)$ for $x\in\Gamma$ where $\nu$ is the outer normal vector to $\Gamma$, such that \[ \|f\|_{C^1} \leq C(\eps,\Gamma) \|\sigma\|_{C^1}. \]
    Moreover, if $\Gamma$ is $C^{2,\aa}$ for some $\aa\in(0,1)$, then we may further estimate \[ \|f\|_{C^{1,\aa}} \leq C(\eps,\Gamma,\aa)\|\sigma\|_{C^{1,\aa}}. \qedhere \]
\end{lemma}
\begin{proof}
    For $l = P(\Gamma)$, define the map $\Phi: [0,l]
    \times(-\eps, \eps) \to \R^2$ via $\Phi(s,t) = \zeta(s) + t\nu(s)$. Then for $\eps$ sufficiently small, $\Phi$ is a $C^1$-diffeomorphism onto a tubular neighborhood of $\Gamma$. Setting $(x(s),t(s)) := \Phi^{-1}(\tilde{\zeta}(s))$, we can define the desired height function $f$ by $f(x(s)) := t(s)$ provided $x$ is monotone, which we will show soon.
    
    In what follows, implicit constants depend only on $\Gamma$ and $\eps$. Let $\tau(s)$ and $\kappa(s)$ denote the unit tangent vector and curvature at $\zeta(s)$. By computing the partial derivatives \begin{align*}
        \partial_s \Phi(s,t) = (1+ t\kappa(s))\tau(s),\qquad \partial_t \Phi(s,t) = \nu(s),
    \end{align*}
    we observe a global bound $1/C\leq D\Phi\leq C$ provided $\eps$ is sufficiently small, where $C$ depends only on $\Gamma$ and $\eps$. Hence the Taylor approximation \begin{align*}
        (x(s),t(s)) = \Phi^{-1}(\tilde{\zeta}(s)) = (s,0) + D\Phi^{-1}(\zeta(s))\cdot \sigma(s) + o(\sigma(s))
    \end{align*}
    implies the estimates \begin{equation}
        \label{sup estimates}
        x(s) = s + O(\|\sigma\|_{C^1}), \qquad t(s) \lesssim \|\sigma\|_{C^1}.
    \end{equation}
    (Note that one can, in fact, obtain the precise bound $|t(s)|\leq |\sigma(s)|$ by the fact that $\tilde{\zeta}(s)$ is the closest point on $\Gamma$ to $\zeta(x(s))$.)

    We now derive estimates for $x'(s)$ and $t'(s)$. By differentiating $\Phi(x(s),t(s)) = \tilde{\zeta}(s)$, we obtain the relation \[ x'(s)[1 + t(s)\kappa(x(s))]\tau(x(s)) + t'(s) \nu(x(s)) = \tau(s) + \sigma'(s). \]
    By \eqref{sup estimates}, we have $\tau(x(s)) = \tau(s) + O(\|\sigma\|_{C^1})$ and $\nu(x(s)) = \nu(s) + O(\|\sigma\|_{C^1})$. Thus we may estimate 
    \begin{align}
        \label{x' estimate}
        x'(s) &= \frac{(\tau(s) +\sigma'(s))\cdot \tau(x(s))}{1 + t(s)\kappa(x(s))} = 1 + O(\|\sigma\|_{C^1})\\
        \label{t' estimate}
        t'(s) &= (\tau(s) + \sigma'(s))\cdot\nu(x(s)) = O(\|\sigma\|_{C^1}).
    \end{align}
    Hence $f$ is well-defined and $f'(x(s)) = \frac{t'(s)}{x'(s)} = O(\|\sigma\|_{C^1}),$ proving the first desired inequality.

    Now implicit constants may additionally depend on $\aa$. To obtain the $C^{1,\aa}$ bound, it suffices to show $[t'(s)/x'(s)]_{C^{\aa}} \lesssim \|\sigma\|_{C^{1,\aa}}$. Because $x'(s)$ is bounded away from 0, we have \begin{align*}
        [t'/x']_{C^\aa} &\lesssim \|t'\|_\infty [x']_{C^\aa} + \|x'\|_\infty[t']_{C^\aa} \lesssim \|\sigma\|_{C^1} [x']_{C^\aa} + [t']_{C^\aa},
    \end{align*} 
    so it is enough to check that $[t']_{C^\aa}\lesssim \|\sigma\|_{C^{1,\aa}}$ and $[x']_{C^\aa} \lesssim 1$. For the former, we denote $\delta(s) := \nu(x(s)) - \nu(s)$ and observe that $\|\delta\|_{C^1}\lesssim \|\sigma\|_{C^1}$ by \eqref{x' estimate}. Since $t'(s) = \sigma'(s) \cdot \nu(x(s)) + \tau(s)\cdot \delta(s)$, we obtain 
    \begin{equation}
        \label{eq:t' aa estimate}
        [t']_{C^\aa} \lesssim [\sigma']_{C^{\aa}} + \|\sigma'\|_\infty [x]_{C^\aa} + \|\delta\|_{C^1} \lesssim \|\sigma\|_{C^{1,\aa}}.
    \end{equation}
    Finally, from \eqref{x' estimate}, we are left to estimate \[ [x']_{C^\aa} \lesssim [(\tau(s) + \sigma'(s))\cdot\tau(x(s))]_{C^\aa} + [t(s)\kappa(x(s))]_{C^\aa}. \]
    A similar estimate as \eqref{eq:t' aa estimate} yields the first term on the righthand side is $O(\|\sigma\|_{C^{1,\aa}})$. For the second term, we remark that the curvature $\kappa$ is $C^\aa$ since $\Gamma$ is $C^{2,\aa}$, so \[ [t(s)\kappa(x(s))]_{C^\aa}\leq \|t\|_\infty [\kappa(x(s))]_\aa + \|\kappa\|_\infty [t]_\aa \lesssim \|t\|_{C^1} \lesssim \|\sigma\|_{C^1} \]
    and hence $[x']_{C^\aa} \lesssim \|\sigma\|_{C^1}$, concluding the proof.
\end{proof}

\begin{lemma}
\label{area and perimeter approx for normal perturbations}
    Let $\phi\in \cM^2(\R^2)$, and let $E_f$ be the normal perturbation of $W_\phi\sb\R^2$ by $f:\partial W_\phi\to \R$, i.e. the bounded region whose boundary is parametrized by the map $u(x) = x + f(x) \nu_{W_\phi}(x)$. Then there exists $c = c(L_\phi, \Lambda_\phi,\|\phi\|_{C^{2,1}(\S^1)})>0$ such that if $\|f\|_{C^1}\leq c$, then  \begin{align}
        \label{area approx}
        |E_f| &= |W_\phi| + \int_{\partial W_\phi} f d\cH^1 + O(\|f\|_{L^\infty}^2) \\
        \label{perimeter approx}
        P_\phi(E_f) &= P_\phi(W_\phi) + \int_{\partial W_\phi} f d\cH^1 + O(\|f\|_{C^1}^2)
    \end{align}
    where the implicit constant also depends only on $L_\phi, \Lambda_\phi$, and $\|\phi\|_{C^{2,1}(\S^1)}$.
\end{lemma}
\begin{remark}
    The linear terms in \eqref{area approx} and \eqref{perimeter approx} may be obtained from standard first variation calculations (e.g., Theorem 17.5 and Proposition 17.8 in \cite{maggi2012}). We provide an alternate derivation to ensure the error depends quadratically on $\|f\|_{C^1}^2$. 
    
\end{remark}
\begin{proof}
For $\theta\in[0,2\pi)$, denote the unit normal and tangent vectors $\nu(\theta) = (\cos \theta, \sin\theta)$ and $\tau(\theta) = (-\sin\theta, \cos\theta)$. For the proof, we will use the Cahn-Hoffman map to parametrize $f$ with respect to $\S^1$ rather than $\partial W_\phi$: \[ u(\theta) = \xi(\theta) + f(\theta)\nu(\theta) = (\gamma + f)\nu + \gamma' \tau. \]
Note that this change is permitted since $\|f\|_{C^1(\S^1)} \sim_{\Lambda_\phi} \|f\circ \xi^{-1}\|_{C^1(\partial W_\phi)}$ by \eqref{speed of Cahn-Hoffman}. Moreover, the desired linear term is now \[ \int_{\partial W_\phi} f\circ \xi^{-1} d\cH^1 = \int_0^{2\pi} f(\gamma + \gamma'')d\theta. \]

The argument of $u$ is given by \[ \omega(\theta) := \arg u(\theta) = \theta + \tan^{-1}\ps{\frac{\gamma'(\theta)}{\gamma(\theta) + f(\theta)}}. \]
We compute \[ \omega' = \frac{(\gamma + f)(\gamma + \gamma'' + f) - \gamma'f' }{(\gamma + f)^2 + (\gamma')^2}. \]
Provided $\|f\|_{C^1} \leq c$ for $c(L_\phi,\Lambda_\phi, \|\phi\|_{C^{2,1}(\S^1)})$ sufficiently small, $\omega$ is monotonic, so we may use polar coordinates and a change of variables to express the area of the perturbed region as \begin{align*}
    |E_f| &= \inv{2} \int_0^{2\pi} |u(\theta)|^2 \omega'(\theta)d\theta\\
    &= \inv{2} \int_0^{2\pi} [(\gamma + f)(\gamma + \gamma'' + f) - f'\gamma']d\theta\\
    &= |W_\phi| + \inv{2} \int_0^{2\pi} [f(2\gamma + \gamma'') - f'\gamma' + f^2]d\theta.
\end{align*}
Via integrating by parts and \eqref{speed of Cahn-Hoffman}, we obtain the desired formula for area: \begin{align*}
    |E_f| = |W_\phi| + \int_0^{2\pi} f(\gamma + \gamma'')d\theta + \inv{2} \int_0^{2\pi}f^2 d\theta.
\end{align*}

Now we show the perimeter approximation. 
Note that we have the Taylor approximation
\begin{align*}
    \frac{u'}{|u'|} = \frac{(\gamma + \gamma'' + f)\tau + f'\nu}{[(\gamma + \gamma'' + f)^2 + (f')^2]^{1/2}} = \tau + \frac{f'}{\gamma + \gamma'' + f}\nu + O(|f'|^2),
\end{align*}
which by a rotation, is equivalent to the estimate \[ \nu_{E_f}(u(x)) = \nu - \frac{f'}{\gamma + \gamma'' + f}\tau + O(|f'|^2). \]
Altogether, we obtain the desired approximation via the area formula and integration by parts: \begin{align*}
    P_\phi(E_f) &= \int_0^{2\pi} \gamma(\nu_{E_f}(u(\theta))) \sqrt{(\gamma + \gamma'' + f)^2 + (f')^2} d\theta\\
    &= \int_0^{2\pi} \bks{\gamma - \frac{f'\gamma'}{\gamma + \gamma'' + f} + O(|f'|^2)}[\gamma + \gamma'' + f + O(|f'|^2)] d\theta\\
    &= P_\phi(W_\phi) + \int_0^{2\pi} (f\gamma - f'\gamma') d\theta + O(\|f\|_{C^1(\S^1)}^2)\\
    &= P_\phi(W_\phi) + \int_0^{2\pi} f(\gamma + \gamma'') d\theta + O(\|f\|_{C^1(\S^1)}^2). \qedhere
\end{align*}
\end{proof}

\noindent\textbf{Proof of \cref{QAT main}:} 

\medskip

Let $\eps_0, C$ be constants depending only on $L_\phi,\Lambda_\phi, \|\phi\|_{C^{2,1}(\S^1)}, M,m$, and let $\eps_0$ be small enough so that the conclusion of \cref{prop:pert} holds; $C$ may change from line to line. We let $E_1, \dots, E_d$ be the (simply connected) components of $E$, and for convenience, we denote $\eps := \|\kappa_E^\phi - \ol{\kappa}_E^\phi\|_{L^2(\partial E)}$. 

By \cref{prop:pert} and \cref{pert into normal pert}, each boundary component $\partial E_j$ is the normal perturbation of a Wulff shape $\partial W_\phi(x_j, r_j)$ by $f_j$ such that $|E_j| = |W_\phi(x_j, r_j)|$, $|r_j - (\ol{\kappa}_E^\phi)^{-1}|\leq C\eps$, and $\|f_j\|_{C^{1,1/2}(\partial W_\phi(x_j, r_j))} \leq C\eps.$ Due to \cref{area and perimeter approx for normal perturbations}, we have the approximation for each $j$ that \begin{equation}
    \label{eq:quadratic control of perimeter}
    |P_\phi(E_j) - P_\phi(W_\phi(x_j, r_j))| \leq C\|f_j\|_{C^1}^2.
\end{equation}
By \eqref{eq:quadratic control of perimeter} and the fact that $E$ has at most $C$ components, we deduce 
\begin{align*}
    \abs{P_\phi(E) - \sum_{j=1}^d P_\phi(r_j W_\phi)} &\leq
    \sum_{j=1}^d |P_\phi(E_j) - P_\phi(r_jW_\phi)| \leq C\sum_{j=1}^d\|f_{j}\|_{C^1(\S^1)}^2 \leq C\eps^2. 
\end{align*}

We now show that the previous estimates on normal perturbations and perimeter approximation are not affected by replacing the individual radii $r_j$ with a uniform radius $r$, chosen so that $E$ has area equal to $F := \cup_{j=1}^d W_\phi(x_j, r_j)$. That is, we set $r$ such that $dr^2 = \sum_{j=1}^d r_j^2$. Since $|r - (\ol{\kappa}^\phi_E)^{-1}|\leq C\eps$, we may repeat the calculations in \eqref{eq:angle almost linear}-\eqref{eq:C^1,1/2 estimate II} and invoke \cref{pert into normal pert} to deduce $\partial E_j$ is the normal perturbation of $\partial W_\phi(x_j, r)$ by $g_j$ such that $\|g_j\|_{C^{1,1/2}(\partial W_\phi(x_j,r))}\leq C\eps$. 

In order to show $|P_\phi(E) - P_\phi(F)| \leq C\eps$, it suffices to check \[ \abs{dr - \sum_{j=1}^d r_j} \leq C\eps^2, \] which amounts to estimating the strictness of Cauchy-Schwarz. Since $1/C \leq r_j, r \leq C$, we have \begin{align*}
    \inv{C}\abs{dr - \sum_{j=1}^d r_j} &\leq d^2r^2 - \ps{\sum_{j=1}^d r_j}^2 = d\ps{\sum_{j=1}^d r_j^2} - \ps{\sum_{j=1}^d r_j}^2 = \sum_{1\leq i < j\leq d} (r_i - r_j)^2.
\end{align*}
Since $|r_j - (\ol{\kappa}_E^\phi)^{-1}| \leq C \eps$ for all $j$, the desired result follows. \hfill$\Box$

\begin{corollary}
\label{QAT cor}
    Let $m,M>0$, $\phi\in \cM^{2,1}(\R^2)$, and let $E\sb\R^2$ be a $C^2$ set such that $|E|=m, P_\phi(E) \leq M$. Then there exists a constant $C = C(L_\phi,\Lambda_\phi, \|\phi\|_{C^{2,1}(\S^1)},m,M)>0$ such that \begin{equation}
        \label{eq:QAT cor}
        \min_{d\in\N} |P_\phi(E) - P_d| \leq C\|\kappa_E^\phi - \ol{\kappa}_E^\phi\|_{L^2(\partial E)}^2
    \end{equation}
    where $P_d := 2\sqrt{|W_\phi|md}$ is the perimeter of $d$ disjoint Wulff shapes of area $m/d$. 
\end{corollary}
\begin{proof}
    If $\|\kappa_E^\phi - \ol{\kappa}_E^\phi\|_{L^2(\partial E)} \leq \eps_0$ with $\eps_0$ as in \cref{QAT main}, then \eqref{eq:QAT cor} follows immediately. Otherwise, if $\|\kappa_E^\phi - \ol{\kappa}_E^\phi\|_{L^2(\partial E)} > \eps_0$, the desired inequality holds with $C = \frac{2M}{\eps_0^2}$:
    \[ \min_{d\in\N} |P_\phi(E) - P_d| \leq 2M \leq \frac{2M}{\eps_0^2}\|\kappa_E^\phi - \ol{\kappa}_E^\phi\|_{L^2(\partial E)}^2. \qedhere \]
\end{proof}

\section{The Flat Flow}
\label{sec:flat flow}

In this section, we study the volume-preserving flat $(\phi,\psi)$-flow in any dimension $N \geq 2$, closely following the minimizing-movements scheme introduced by \cite{Almgren1993} and adapted by \cite{mugnai2016} and \cite{Kim2020star} to the isotropic volume-preserving case. 

We fix an initial bounded set $E_0\sb\R^N$ with volume $m:= |E_0|>0$. Given a bounded set $F\sb\R^N$ and a time step $h>0$, we consider the energy functional \begin{equation}
    \label{energy functional}
    \cF_h(E,F) := P_\phi(E) + \inv{h} \int_E \sd^\psi_F dx + \inv{\sqrt{h}}||E|- m |.
\end{equation}
Note that $\cF_h(\cdot, F)$ has a minimizer among bounded sets of finite perimeter by a standard compactness argument (e.g., see \cite[Lemma 3.1]{mugnai2016}). 

An \textit{\textbf{approximate flat $(\phi,\psi)$-flow}} is defined by setting $E^{(h)}_0 := E_0$ and iteratively letting $E^{(h)}_{k+1}$ minimize $\cF_h(\cdot, E^{(h)}_k)$. With respect to time, we denote $E^{(h)}(t) := E^{(h)}_{\floor{t/h}}$ for $t\in[0,\infty)$. 
The flat flow is then obtained by taking an appropriate subsequence as $h\to 0$.

\begin{theorem}
    \label{existence of flat flow}
    Let $\phi\in\cM^{2,1}(\R^N)$, $\psi\in\cM(\R^N)$, $E_0\sb\R^N$ be a bounded set of finite perimeter, and for all $h>0$, let $\{\Eh(t)\}_{t\geq0}$ be an approximate flat flow starting at $E_0$. Then there exists a family of sets of finite perimeter $\{E(t)\}_{t\geq0}$, which we call a \textbf{volume-preserving flat $(\phi,\psi)$-flow} starting at $E_0$, such that $E(0)=E_0$ and for some subsequence $h_n\to 0$, we have $\lim_{n\to \infty} |E^{(h_n)}(t)\Delta E(t)| = 0$ for all $t\geq0$. Moreover, for all $0\leq s\leq t$ and a constant $C = C(N,L_\phi,L_\psi)$, \begin{align}
        \label{Holder continuity in time}
        |E(s)\Delta E(t)| &\leq C|s-t|^{1/2}\\
        \label{volume preserved}
        |E(t)| &= m\\
        \label{perimeter monotonicity}
        P_\phi(E(t)) &\leq P_\phi(E_0)
    \end{align}
    where $L_\phi$ and $L_\psi$ are given in \eqref{eq:lphi}.
\end{theorem}

Before proving \cref{existence of flat flow}, we will need a number of technical lemmas. Some of the arguments for \eqref{energy functional} are parallel to the ones in \cite{mugnai2016} with instances of the isotropic perimeter functional replaced by $P_\phi$. However, to the best of our knowledge, the statements specifically for \eqref{energy functional} in the anisotropic setting are not available. Therefore, for the sake of completeness, we write it out in full in Section~\ref{sec:sta}. Some of the standard estimates and proofs are deferred to Appendix~\ref{ap:flat}. In Sections~\ref{sec:ell} and \ref{sec:thm41}, we highlight cases where the anisotropy introduces a nontrivial obstacle.

In \cref{existence of flat flow}, we assume $\phi\in\cM^{2,1}(\R^N)$, restricting to this level of regularity for two reasons. The first is to invoke a priori regularity of $(\Lambda, r_0)$-minimizers of $P_\phi$, per \cite{philippis2014}, and the second is to apply Schauder estimates in \cref{prop:elliptic reg} below.

\subsection{Standard estimates}
\label{sec:sta}

The dissipation of two bounded sets $E$ and $F$ is defined as 
    \[ \cD^\psi(E,F) := \int_{E\Delta F} d^\psi_F dx. \]
Due to the relation, \[ \int_E \sd^\psi_F dx = \int_{E\Delta F} d^\psi_F dx - \int_F d^\psi_F dx, \]
we see that replacing the term $\int_E \sd^\psi_F dx$ in the definition of $\cF_h$ with $\cD^\psi(E,F)$ makes no difference to the minimization problem. Hence we obtain the following dissipation inequality:
\begin{equation}
    \label{dissipation}
    P_\phi(E_{k+1}^{(h)}) + \inv{h} \cD^\psi(E_{k+1}^{(h)}, E_k^{(h)}) + \inv{\sqrt{h}}||E_{k+1}^{(h)}|-m| \leq P_\phi(E_{k}^{(h)}) + \inv{\sqrt{h}}||E_{k}^{(h)}|-m|.
\end{equation}
By iterating \eqref{dissipation}, we further get for all $k\geq0$: \begin{equation}
    \label{iterated dissipation}
    P_\phi(\Eh_k) + \inv{h} \sum_{i=0}^{k-1} \cD^\psi(\Eh_{i+1},\Eh_i) + \inv{\sqrt{h}}||\Eh_{k}| - m| \leq P_\phi(E_0).
\end{equation}

The following two lemmas are standard estimates adapted from \cite{mugnai2016}; their proofs can be found in Appendix~\ref{ap:flat}. 
\begin{lemma}
    \label{lem:standard estimate flat flow}
    Let $F\sb\R^N$ be a bounded set of finite perimeter and let $E$ be a minimizer of $\cF_h(\cdot, F)$. Then there exist constants $c,C>0$ depending only on $N,L_\phi,L_\psi$ such that
    \begin{itemize}
        \item[(i)] ($L^\infty$ estimate)  \begin{equation}
            \label{eq:Linfty estimate}
            \sup_{E\Delta F} d^\psi_F \leq c\sqrt{h}
        \end{equation}
        \item[(ii)] ($L^1$ estimate) For all $\ell\leq c\sqrt{h}$, \begin{equation}
            \label{eq:L1 estimate}
            |E\Delta F|\leq C\ps{\ell P_\phi(E) + \inv{\ell} \cD^\psi(E,F)}
        \end{equation}
        \item[(iii)] ($L^2$ estimate) \begin{equation}
            \label{eq:L2 estimate}
            \int_{\partial^* E} (d^\psi_F)^2 d\cH^{N-1} \leq C\cD^\psi(E,F).
        \end{equation}
        where $\partial^* E$ is the reduced boundary of $E$.
    \end{itemize}
\end{lemma}

\begin{lemma}[H\"older continuity in time]
    \label{lem:Holder continuity in time}
    Let $\Eh(t)$ be an approximate flat $(\phi,\psi)$-flow from initial set $E_0$. Then there exists a constant $C = C(N,L_\phi,L_\psi)$ such that for all $0\leq s\leq t<\infty$,
        \begin{equation}
            \label{eq:approx Holder continuity in time}
            |\Eh(s)\Delta \Eh(t)| \leq CP_\phi(E_0)\max\{|t-s|,h\}^{1/2}.
        \end{equation}
\end{lemma}

\subsection{Elliptic regularity}
\label{sec:ell}

\begin{definition}
\label{def:almost minimizer}
Let $\phi$ be a norm on $\R^N$. We say a bounded set $E\sb\R^N$ of finite perimeter is a $(\Lambda,r_0)$-minimizer of $P_\phi$ if for all $r<r_0$ and $F\sb\R^N$ such that $E\Delta F \sb\sb B_r(x)$, we have \begin{equation}
    \label{almost minimizer inequality}
    P_\phi(E;B_r(x)) \leq P_\phi(F;B_r(x)) + \Lambda |E\Delta F|.
\end{equation} 
\end{definition}

We show that a minimizer of $\cF_h(\cdot, F)$ is a $(\Lambda, r_0)$-minimizer of $P_\phi$.

\begin{lemma}
    \label{flat flow is almost minimizer}
    Let $F\sb\R^N$ be a bounded set of finite perimeter and let $E$ be a minimizer of $\cF_h(\cdot, F)$. Then $E$ is a $(\Lambda, r_0)$-minimizer of $P_\phi$ for any $r_0>0$ and $\Lambda = \frac{c + 1}{\sqrt{h}} + \frac{2L_\psi r_0}{h}$, where $c$ is the constant from \cref{lem:standard estimate flat flow}(i).
\end{lemma}
\begin{proof}
    Fix $x\in\R^N$, and suppose $G\sb\R^N$ is such that $E\Delta G \sb\sb B_r(x)$ where $r<r_0$. By minimality, we have \begin{align*}
        P_\phi(E) &\leq P_\phi(G) + \inv{h}\int_{G\Delta F} d^\psi_F + \inv{\sqrt{h}}||G|-m| - \ps{\inv{h}\int_{E\Delta F} d^\psi_F + \inv{\sqrt{h}}||E|-m|}\\
        &\leq P_\phi(G) + \inv{h} \int_{E\Delta G} d^\psi_F + \inv{\sqrt{h}}|E\Delta G|.
    \end{align*}
    By \cref{lem:standard estimate flat flow}(i), we have $\sup_{E\Delta F} d^\psi_F \leq c\sqrt{h}$ and hence $\sup_{E\Delta G} d^\psi_F \leq c\sqrt{h} + 2L_\psi r_0$. Thus \[ P_\phi(E) \leq P_\phi(G) + \ps{\frac{c + 1}{\sqrt{h}} + \frac{2L_\psi r_0}{h}}|E\Delta G|. \]
    Since $P_\phi(E) - P_\phi(G) = P_\phi(E;B_r(x)) - P_\phi(G;B_r(x))$, we are done.
\end{proof}

\begin{lemma}[Euler-Lagrange equation]
\label{lem:el}
Let $\phi\in \cM^2(\R^N), \psi\in\cM(\R^N)$ and let $F\sb\R^N$ be a bounded set of finite perimeter and let $E$ be a minimizer of $\cF_h(\cdot, F)$. Then $E$ satisfies the Euler-Lagrange equation 
\begin{equation}
    \label{eq:euler-lagrange}
    \frac{\sd^\psi_F}{h} = -\kappa_E^\phi + \lambda \qquad \text{on }\partial^*E
\end{equation}
for some Lagrange multiplier $\lambda\in\R$.
Moreover if $|E| \neq m$, then $\lambda = \inv{\sqrt{h}} \sgn(m - |E|)$. 
\end{lemma}

In light of the Euler-Lagrange equation, for any approximate flat $(\phi,\psi)$-flow $\Eh(t)$, we define the discrete velocity \[ \vh(t,x) := \begin{cases}
    \inv{h}\sd^\psi_{\Eh(t-h)}(x) &t\in[h,\infty)\\
    0 &t\in[0,h)
\end{cases} \]
and set $\lambdah(t)$ to be the Lagrange multiplier for $\Eh(t)$.

Lemma~\ref{flat flow is almost minimizer} and  Lemma~\ref{lem:el} yield the elliptic regularity:

\begin{proposition}[Elliptic regularity]
    \label{prop:elliptic reg}
    Let $\phi\in \cM^{2,1}(\R^N)$, let $F\sb\R^N$ be a bounded set of finite perimeter, and let $E$ be a minimizer of $\cF_h(\cdot, F)$. Then there exists a relatively closed singular set $\Sigma\sb \partial E$ such that $\cH^{N-3}(\Sigma)=0$ and $\partial E\setminus \Sigma$ is locally a $C^{2,\aa}$ manifold for any $\aa\in(0,1)$.
\end{proposition}
\begin{remark}
    When $N\leq 3$, $\Sigma$ is empty, so $\partial E$ is simply a $C^{2,\aa}$ manifold.
\end{remark}

\begin{proof}
By \cref{flat flow is almost minimizer}, $E$ is a $(\Lambda,r_0)$-minimizer of $P_\phi$, so by \cite{philippis2014}, $\partial E$ is locally a $C^1$ manifold outside of a singular set $\Sigma$ with $\cH^{N-3}(\Sigma)=0$. We need only upgrade $C^1$ regularity to $C^{2,\aa}$.

Fix a regular point $x_0\in \partial E\setminus \Sigma$, and let $u:\R^N\to\R$ be a $C^1$ function such that in a neighborhood $U$ of $x_0$, $\partial E\cap U = \set{x\in U: u(x)=0}$ and $\nabla u$ does not vanish on $U$. Assuming a priori regularity of $\nu_E = \frac{\nabla u}{|\nabla u|}$, we can express the Euler-Lagrange equation \eqref{eq:euler-lagrange} on $\partial E\cap U$ as \begin{equation}
    \label{eq:elliptic PDE}
    \kappa_E^\phi = \div_\tau(\nabla \phi(\nu_E)) = \tr\ps{D^2\phi\ps{\frac{\nabla u}{|\nabla u|}} \bks{\frac{D^2u}{|\nabla u|} - \frac{\nabla u\otimes\nabla u}{|\nabla u|^3}}} = -\frac{\sd^\psi_F}{h} + \lambda.
\end{equation} 
Recalling that $\Lambda_\phi^{-1} I \leq D^2\phi|_{\S^{N-1}} \leq \Lambda_\phi I$, we see that \eqref{eq:elliptic PDE} is an elliptic equation in non-divergence form with continuous coefficients, and it follows from standard Schauder estimates (e.g., \cite[Proposition 2.31]{Fernandez_Real_2022}) that $u\in C^{1,\aa}$ for any $\aa\in(0,1)$. Since $D^2\phi$ and $\sd^\psi_F$ are Lipschitz, we now obtain an elliptic equation with $C^\aa$ coefficients, and applying Schauder estimates once more yields that $u\in C^{2,\aa}$, completing the proof.
\end{proof}

\subsection{Proof of \cref{existence of flat flow}}
\label{sec:thm41}

We are almost in a position to prove \cref{existence of flat flow}, but we must first ensure that the approximate flow remains bounded in finite time in order to obtain compactness. In addition, some control over the Lagrange multiplier $\lambdah(t)$ for $\Eh(t)$ is needed in the proof of \cref{thm:exponential convergence of flat flow} to limit the amount of time during which the volume restraint $|\Eh(t)|=m$ is inactive. Both of these issues are resolved in \cref{lem:boundedness in finite time}.

\begin{lemma}
    \label{lem:regularity at barrier}
    Let $\phi\in \cM^{2,1}(\R^N)$ and suppose $E$ is a $(\Lambda, r_0)$-minimizer of $P_\phi$ which is contained in a half-space $H\sb\R^N$. Then any point $x\in \partial E\cap \partial H$ must be a regular point of $E$. 
\end{lemma}
\begin{proof}
    Without loss of generality, let $x=0$. We observe that for any $r>0$, the rescalings $E_{r} := r^{-1}E$ are $(\Lambda r, r_0/r)$-minimizers of $P_\phi$ and are also contained in $H$. By applying \cite[Lemma 2.1]{philippis2014} with a diagonalization, there exists $r_j\to 0$ such that $E_{r_j}$ converges in $L^1_{loc}$ to some set $\tilde{E}\sb H$ which is an absolute $P_\phi$-minimizer, i.e. it is a $(0, \infty)$-minimizer of $P_\phi$. Then \cite[Theorem 1.2]{SSA77} implies that 0 is a regular point of $\partial \tilde{E}$. By \cite[Lemma 2.1]{philippis2014}, we have that 0 is a regular point of $\partial E_{r_j}$ for sufficiently large $j$, and hence also of $\partial E$ by rescaling. 
\end{proof}

Note that for $\phi\in \cM^2$, $\kappa^\phi$ satisfies a monotonicity property: 
\begin{equation}
    \label{curvature monotonicity}
    \text{if $E,F\sb\R^N$ are $C^2$ such that $F\sbq E$, then for any point $x\in \partial E\cap \partial F$, we have $\kappa^\phi_E(x) \leq \kappa^\phi_F(x)$.}
\end{equation}
This is proven, for instance, in \cite{Chambolle2015} for a general class of perimeter functionals. We observe as a consequence that if $E$ satisfies the interior $rW_\phi$-property at $x\in\partial E$ and $\partial E$ is $C^2$ near $x$, then $\kappa_E^\phi(x) \leq \inv{r}$. Indeed, this follows by applying \eqref{curvature monotonicity}, which is a local statement, to $F = W_\phi(x-r\nu_E(x),r)$. 

\begin{proposition}
[Boundedness in finite time]
    \label{lem:boundedness in finite time}
    Let $\{\Eh(t)\}_{t\geq0}$ be an approximate flat $(\phi,\psi)$-flow and fix $T>0$. There exist $h_0 = h_0(m,P_\phi(E_0))>0$  and $r_T = r_T(N,L_\phi,L_\psi,m, P_\phi(E_0),\diam(E_0),T)$ such that $\Eh(t) \sb r_TW_\phi$ for all $t\in[0,T]$ and $h\leq h_0$.
\end{proposition}
\begin{proof}
    In what follows, $C$ depends only on $N,L_\phi,L_\psi$ and may change from line to line. Fix $h>0$, and for any $t\in[0,T]$ define \[ r_t := \inf\{r>0: \Eh(t)\sb rW_\phi \}. \]
    Choose $x\in \partial \Eh(t)\cap  \partial(r_t W_\phi)$. There exists a half-space $H$ containing $\Eh(t)$ and satisfying $x \in \partial H$. By \cref{lem:regularity at barrier}, $x$ is a regular point. Then we can apply the monotonicity principle \eqref{curvature monotonicity} to deduce $\kappa^\phi_{\Eh(t)}(x) \geq \inv{r_t}$, and hence by the Euler-Lagrange equation, $\vh(t,x) = -\kappa^\phi_{\Eh(t)}(x) + \lambdah(t) \leq \lambdah(t).$ Thus $r_t \leq r_{t-h} + C\lambdah(t)h$ and by iterating,
    \begin{equation}
        \label{eq:radius bound}
        r_\tau \leq r_0 + C\int_0^\tau |\lambdah(t)|dt.
    \end{equation}
    To bound the Lagrange multiplier with respect to $r_t$, we invoke the divergence theorem: letting $X$ be a $C^1_c$ vector field such that $X(x)\equiv x$ in $r_tW_\phi$, we obtain the estimate 
    \begin{align*}
        n|\lambdah(t)||\Eh(t)| &= \abs{\int_{\Eh(t)} \lambdah(t) \div X dx} = \abs{\int_{\partial^* \Eh(t)} \lambdah(t) \nu\cdot X d\cH^{N-1}} \\
        &= \abs{\int_{\partial^*\Eh(t)} (\kappa_{\Eh(t)}^\phi + \vh)\nu \cdot X d\cH^{N-1}}\\
        &\leq \abs{\int_{\partial^*\Eh(t)} \div_{\tau,\phi}X \,dP_\phi} + \abs{r_t\int_{\partial^*\Eh(t)} \vh d\cH^{N-1}} \\
        &\leq (N-1) P_\phi(\Eh(t)) + r_t \ps{\int_{\partial^*\Eh(t)} \vh(t,x)^2 d\cH^{N-1}}^{1/2}
    \end{align*}
    where in the last step we used that $\div_{\tau,\phi}X = N - \nabla\phi(\nu)\cdot \frac{\nu}{\phi(\nu)}= N-1$ and the Cauchy-Schwarz inequality.
    Note that for $h\leq h_0(m,P_\phi(E_0))$ sufficiently small, $|\Eh(t)| \geq \frac{m}{2}$ by dissipation. Hence, by integrating the previous estimate in time and applying \cref{prop:l2}, we obtain

    \begin{align}
        \label{eq:lagrange mult L1 bound}
        \int_0^\tau |\lambdah(t)|dt &\leq \frac{2}{m}\ps{P_\phi(E_0)\tau + \frac{1}{n}\int_0^\tau r_t\ps{\int_{\partial^*\Eh(t)} \vh(t,x)^2 d\cH^{N-1}}^{1/2}dt}\notag\\
        &\leq \frac{2}{m} P_\phi(E_0)\tau + \frac{C}{m} P_\phi(E_0)^{1/2}\ps{\int_0^\tau r_t^2 dt}^{1/2}\notag \\
        &\leq C'\ps{\tau + \ps{\int_0^\tau r_t^2 dt}^{1/2}}
    \end{align}
    where $C'$ depends only on $N,L_\phi, L_\psi,m,P_\phi(E_0)$. Altogether, by \eqref{eq:radius bound} and \eqref{eq:lagrange mult L1 bound} we have \begin{equation}
        \label{eq:radius bound 2}
        r_\tau \leq r_0 + C'\tau + C'\ps{\int_0^\tau r_t^2 dt}^{1/2}.
    \end{equation}
    The desired boundedness now follows from using integrating factors as argued in \cite[Lemma 3.8]{mugnai2016}. 
\end{proof}

\noindent\textbf{Proof of \cref{existence of flat flow}:}

\medskip

    By \cref{lem:boundedness in finite time}, we may find a sequence $h_n\to 0$ such that for all $q\in \Q_{\geq0}$, $\Ehn(q)$ converges in $L^1$ to some $E(q)$. We can check that the H\"older continuity in time \eqref{approx Holder continuity in time} is preserved in the limit: for $p,q\in\Q_{\geq0}$, \begin{equation}
        \label{eq:Holder conintuiy in rational time}
        |E(p)\Delta E(q)| \leq \lim_{n\to\infty} \ps{|E(p)\Delta \Ehn(p)| + |\Ehn(q) \Delta E(p)| + C\max\{h_n, |p-q|^{1/2}\}} = C|p-q|^{1/2}.
    \end{equation}
    
    We can then define the entire flat $(\phi,\psi)$-flow via continuity in $L^1$: for $t\geq0$, \begin{equation*}
        \label{eq:flat flow from continuity}
        E(t) := \lim_{\substack{q\to t\\q\in \Q_{\geq0}}} E(q),
    \end{equation*}
    for which well-definedness is standard to check using \eqref{eq:Holder conintuiy in rational time}. Moreover the H\"older continuity in time is again preserved: for $0\leq s\leq t$ and $p,q\in\Q_{\geq0}$, we have 
    \begin{align*}
        |E(s)\Delta E(t)| \leq |E(s)\Delta E(p)|  + |E(q) \Delta E(t)| + C|p-q|^{1/2}
    \end{align*}
    and sending $p\to s$ and $q\to s$, we obtain $|E(s)\Delta E(t)| \leq C|s-t|^{1/2}$. By dissipation, we have for all $t\geq0$ \[ 
 P_\phi(\Ehn(t)) + \inv{\sqrt{h_n}}||\Ehn(t)| - m| \leq P_\phi(E_0) \]
    and by sending $n\to\infty$, we obtain $|E(t)|=m$ and $P_\phi(E(t))\leq P_\phi(E_0)$ by the lower semicontinuity of anisotropic perimeter. 
    \hfill$\Box$

\subsection{Further estimates}

We conclude this section with the following corollaries, which will be used in the subsequent section, particularly in the proofs of \cref{lem:dissipation exponential decay} and \cref{thm:exponential convergence of flat flow}.

\begin{corollary}
    \label{cor:L2 control lagrange mult}
    Let $\{\Eh_t\}_{t\geq0}$ be an approximate flat $(\phi,\psi)$-flow and fix $T>0$. There exists constants $h_0 = h_0(m,P_\phi(E_0))$ and $C = C(N,L_\phi,L_\psi, m,P_\phi(E_0),\diam(E_0),T)$ such that for $h\leq h_0$, \[ \int_0^T |\lambdah(t)|^2 dt \leq C \] and \[ |\{t\in[0,T]: |\Eh(t)| \neq m\}| \leq Ch. \]
\end{corollary}
\begin{proof}
    The first inequality follows from \eqref{eq:lagrange mult L1 bound} in the proof of \cref{lem:boundedness in finite time} and Cauchy-Schwarz inequality. For the latter inequality, we observe that $|\lambdah(t)| = \inv{\sqrt{h}}$ whenever $|\Eh(t)| \neq m$, and hence \[ |\{t\in[0,T]: |\Eh(t)| \neq m\}| \leq \int_0^T h |\lambdah(t)|^2dt \leq Ch. \qedhere \] 
\end{proof}

We will also need the following distance comparison result, that if $E$ is close to a disjoint union $F$ of Wulff shapes, then the flat flow from $E$ remains close to $F$ for a short amount of time. Its proof is parallel to the argument from the isotropic case given in \cite[Lemma 4.3]{julin2020} as a consequence of \cref{cor:L2 control lagrange mult} and the monotonicity principle \eqref{curvature monotonicity}. 

\begin{corollary}[Distance Comparison Result]
    \label{cor:distance comparison}
    Suppose $E\sb\R^N$, $P_\phi(E) \leq C_0$, and $F = \cup_{i=1}^N W_\phi(x_i, r)$ is a disjoint union of Wulff shapes such that $\inv{C_0}\leq r \leq C_0$. Then there exists $\eps_0, h_0, C>0$ depending only on $C_0, N, L_\phi, L_\psi$ such that if $\sup_{E\Delta F} d^\psi_F \leq \eps \leq \eps_0$ and $h\leq \min\{\sqrt{\eps}, h_0\}$, then 
    \begin{equation}
        \label{eq:distance comparison}
        \sup_{\Eh(t)\Delta F} d^\psi_F \leq C\eps^{1/9}
    \end{equation}
    for all $t < \sqrt{\eps}$.
\end{corollary}

\section{Long-term Convergence}
\label{sec:lon}
We are now ready to prove \cref{thm:exponential convergence of flat flow}. 
The proof runs parallel to that of \cite[Theorem 1.2]{julin2022}; we include it for completion's sake and to highlight the applications of \cref{QAT main}, which are twofold. First, we apply the quantitative Alexandrov theorem in the form of \cref{QAT cor} to obtain exponential decay of dissipations via a Gr\"{o}nwall type argument, from which we can deduce the exponential convergence of the flat flow in $L^1$. Then, we apply \cref{QAT main} to show the approximate flat flow is an exponentially small perturbation of a disjoint union of Wulff shapes.

One might hope that the rate of exponential convergence asserted in \cref{thm:exponential convergence of flat flow} is uniform with respect to basic geometric properties of the initial set $E_0$, particularly $|E_0|, P_\phi(E_0)$ and $\diam(E_0)$. However, as pointed out by \cite{julin2022}, this is not true if one considers in the isotropic case the example of two disks of different radii. The smaller disk disappears in finite time, but as the initial radii are closer to equal, it takes arbitrarily long for the convergence to occur.

In \cref{thm:exponential convergence of flat flow}, we show that the rate of exponential convergence is uniform once the limiting energy $p(t)$ defined in \eqref{eq:p} is close enough to converging. To be more precise, we will show that if 
\begin{equation}
    \label{eq:T0}
    P_d < p(t) < P_{d+1}-\delta\quad \text{for all }t\geq T_0
\end{equation}
for some $d\in\N$ and $T_0,\delta>0$, then $(E(t))_{t\geq T_0}$ converges exponentially at a rate which depends only on $L_\phi, \Lambda_\phi, \|\phi\|_{C^{2,1}(\S^1)},L_\psi, |E_0|, P_\phi(E_0), \diam(E_0)$ and $\delta$. This highlights that the main obstruction to obtaining a uniform rate of convergence is that $T_0$ may depend rather arbitrarily on the flat flow $E(t)$. 

\medskip

\noindent\textbf{Proof of \cref{thm:exponential convergence of flat flow}:}

\medskip

Let $E_0\sb\R^2$ be a bounded set of finite perimeter such that $|E_0|=m$. Let $E(t)$ be a volume-preserving flat $(\phi,\psi)$-flow starting from $E_0$, and $\Ehn(t)$ a corresponding sequence of approximate flows. We will frequently pass to a subsequence of $h_n$ without relabeling. By dissipation, we note that the energy \[ P_\phi(\Ehn(t)) + \inv{\sqrt{h_n}}||\Ehn(t)| - m| \]
is decreasing in $t$ and uniformly bounded by $P_\phi(E_0)$. Thus, we may pass to a subsequence such that the following pointwise limit exists: 
\begin{equation}
    \label{eq:p}
    p(t) := \lim_{n\to\infty} \ps{P_\phi(\Ehn(t)) + \inv{\sqrt{h_n}}||\Ehn(t)| - m|}. 
\end{equation}

We assume that $p(t)$ is not eventually constant (the alternative can be argued exactly as in Case 2 of \cite[Theorem 1.2]{julin2022}). Recall that $P_d := 2\sqrt{|W_\phi|md}$ is the $\phi$-perimeter of $d$ disjoint Wulff shapes of area $m/d$. Then there exists $d\in\N$ such that $\lim_{t\to\infty} p(t) \in [P_d, P_{d+1})$, and moreover there exists $T_0>0$ such that \eqref{eq:T0} is satisfied for $\delta := \inv{2}(P_{d+1} - \lim_{t\to\infty} p(t))$.

\begin{lemma}[Exponential Decay in Dissipations]
    \label{lem:dissipation exponential decay}
    Suppose $T_0>0$ and $\delta>0$ satisfy \eqref{eq:T0}.
    Then for any $T> T_0$, there exists $n_T \in \mathbb{N}$ such that for all $n\geq n_T$ and $t\in[T_0,T]$, 
    \begin{equation}
        \label{eq:dissipation exponential decay}
        \inv{h_n} \sum_{j=\floor{t/h_n}}^{\floor{T/h_n}-1} \cD^\psi(\Ehn_{j+1}, \Ehn_j) \leq 2P_\phi(E_0)e^{-(t-T_0)/C_0}
    \end{equation}
    where $C_0 = C_0(L_\phi, \Lambda_\phi, \|\phi\|_{C^{2,1}(\S^1)},L_\psi,m,P_\phi(E_0),\diam(E_0),\delta)$.
    Moreover, for $s>t\geq T_0$, \begin{equation}
        \label{eq:exp L1 conv}
        |E(s)\Delta E(t)| \leq C_0e^{-(t-T_0)/2C_0}.
    \end{equation}
\end{lemma}

\begin{remark}
    We note that the estimate  \eqref{eq:exp L1 conv} (which is of the flat flow and not of the approximate flat flow) implies exponential convergence to some set $E_\infty\sb\R^2$ in $L^1$. It remains to characterize $E_\infty$ as a disjoint union of Wulff shapes and show that exponential convergence is also achieved with respect to the distance between the boundaries. 
\end{remark}

\begin{proof}[Proof of \cref{lem:dissipation exponential decay}]
    In what follows, $C_0$ is a constant that may change from line to line and whose dependencies are as aforementioned. From \eqref{eq:T0}, for $n\geq n_T$ sufficiently large and for all $t\in[T_0,T]$, we have 
    \begin{align}
    \label{eq:exp1}
        P_\phi(\Ehn(t)) + \inv{\sqrt{h_n}}||\Ehn(t)| - m| \in (P_d, P_{d+1}-\delta).
    \end{align}
    
    We claim that for all $t\in[T_0,T]$ and $n\geq n_T$ such that $|\Ehn(t)|=m$, we have the estimate
    \begin{align}
    \label{eq:exp3}
        \sum_{j=\floor{t/h_n}}^{\floor{T/h_n}-1} \cD^\psi(\Ehn_{j+1}, \Ehn_j) &\leq 
            \frac{C_0}{h_n} \cD^\psi(\Ehn(t), \Ehn(t-h_n)).
    \end{align}
    Once we obtain \eqref{eq:exp3}, the proof of \eqref{eq:exp1} is given as follows. Let $D(t) := \inv{h_n} \sum_{j=\floor{t/h_n}}^{\floor{T/h_n}-1} \cD^\psi(\Ehn_{j+1},\Ehn_j)$, so that \eqref{eq:exp3} can be equivalently expressed as
    \begin{equation}
        \label{eq:exp4}
        D(t) \leq \frac{C_0}{h_n} [D(t-h_n) - D(t)] \qquad \hbox{ and thus } \qquad D(t) \leq \ps{1 - \inv{1+C_0/h_n}}D(t-h_n).
    \end{equation}
    Recall from \cref{cor:L2 control lagrange mult} that 
    \begin{equation}
        \label{eq:inactive volume}
        |\{\lfloor T_0/h_n\rfloor \leq i \leq \lfloor T/h_n\rfloor: |\Ehn_i|\neq m\}| \leq C' = C'(L_\phi,L_\psi,m,P_\phi(E_0),\diam(E_0),T),
    \end{equation}
    which means \eqref{eq:exp3} fails for at most $C'$ many timesteps $i$. Therefore, starting at any $t\in[T_0,T]$ (not necessarily satisfying $|\Ehn(t)|=m$), we may iterate the second inequality in \eqref{eq:exp4} at least $\floor{t/h_n}-\floor{T_0/h_n}-C'$ many times to obtain the estimate
    \begin{align}
        D(t) &\leq D(T_0)\ps{1 - \inv{1+C_0/h_n}}^{\floor{t/h_n}-\floor{T_0/h_n}-C'}.
    \end{align}
    By using dissipation to bound $D(T_0)\leq P_\phi(E_0)$ and taking $n_T$ sufficiently large, we obtain \eqref{eq:dissipation exponential decay}. 
    
    It remains to prove the claim \eqref{eq:exp3}. For all $t\in [T_0,T]$ satisfying the volume constraint $|\Ehn(t)|=m$, we may use the iterated dissipation inequality as in \eqref{iterated dissipation}:
    \begin{align}
        \label{eq:exponential 1}
        \inv{h_n} \sum_{j=\floor{t/h_n}}^{\floor{T/h_n}-1} \cD^\psi(\Ehn_{j+1}, \Ehn_j) &\leq P_\phi(\Ehn(t)) - P_d.
    \end{align}
    By \eqref{eq:exp1}, we have $\delta < P_{d+1} - P_\phi(\Ehn(t)) < P_{d+1}-P_d$, which yields
    \begin{align}
    \label{eq:exp2}
        P_\phi(\Ehn(t)) - P_d \leq \frac{P_{d+1}-P_d}{\delta} \min_{k\in \N} |P_\phi(\Ehn(t)) - P_k |.
    \end{align}
    Using the inequalities \eqref{eq:exp1} and \eqref{eq:exp2}, we apply the quantitative Alexandrov theorem in the form of \cref{QAT cor} to bound
    \begin{align}
    \label{eq:exp5}
        \inv{h_n} \sum_{j=\floor{t/h_n}}^{\floor{T/h_n}-1} \cD^\psi(\Ehn_{j+1}, \Ehn_j) &\leq C\|\kappa_{\Ehn(t)}^\phi - \ol{\kappa}_{\Ehn(t)}^\phi\|_{L^2(\partial \Ehn(t))}^2.
    \end{align}
    Then, our claim follows from \eqref{eq:exp5} and Lemma~\ref{lem:expu} below.

    On the other hand, the proof of \eqref{eq:exp L1 conv} follows from a standard application of \cref{lem:standard estimate flat flow}(ii) and \eqref{eq:dissipation exponential decay} as in the proof of \cite[Theorem 1.2]{julin2022}.
\end{proof}

\begin{lemma}
\label{lem:expu}
    Suppose $E\sb\R^N$ is a minimizer of $\cF_h(\cdot, F)$. Then for $C = C(N,L_\phi,L_\psi)$, \begin{equation}
        \|\kappa^\phi_E - \ol{\kappa}^\phi_E\|_{L^2(\partial^* E)} \leq \frac{C}{h^2} \cD^\psi(E,F).
    \end{equation}
\end{lemma}
\begin{proof}
    Observe that for any $f\in L^2(\partial^* E)$, the quantity $$\min_{c\in\R} \|f-c\|_{L^2(\partial^* E)}$$ is attained at $c = \inv{\cH^1(\partial^* E)} \int_{\partial^*E} fd\cH^{N-1}$. Thus, recalling the Euler-Lagrange equation $\frac{\sd^\psi_F}{h} = -\kappa^\phi_E + \lambda$ and using  \cref{lem:standard estimate flat flow}(iii), we may estimate \begin{equation*}
        \|\kappa_E^\phi - \ol{\kappa}_E^\phi\|_{L^2(\partial E)} \leq \|\kappa_E^\phi - \lambda\|_{L^2(\partial E)} = \inv{h^2} \int_{\partial F} (d^\psi_F)^2dx \leq \frac{C}{h^2} \cD^\psi(E,F). \qedhere
    \end{equation*}
\end{proof}

    In what follows, $C$ is a constant that does not depend on time or the timestep $h_n$, and $C_0$ denotes the constant from \cref{lem:dissipation exponential decay} and will remain static. Moreover, $T_1\geq T_0$ will be a time which may change from line to line but will be such that the increment $T_1 - T_0$ depends only on $L_\phi, \Lambda_\phi, \|\phi\|_{C^{2,1}(\S^1)},L_\psi,m,P_\phi(E_0),\diam(E_0),\delta$.
    
    For given $t\geq T_1$, we show that there exists $n_t \in \mathbb{N}$ such that for all $n \geq n_t$ and for some $t_n\in [t-e^{-t/4C_0},t]$, 
    \begin{align}
    \label{eq:11}
        k_n(t_n) \leq e^{-t/4C_0} \hbox{ and } |\Ehn(t_n)|=m
    \end{align} 
    where we denote $k_n(s) := \|\kappa_{\Ehn(s)}^\phi - \ol{\kappa}_{\Ehn(s)}^\phi\|_{L^2(\partial\Ehn(s))}$ for $s \geq 0$. Using Markov's inequality, for any $n\geq n_t$ we may estimate
    \begin{align*}
        \big|\{s\in[t-e^{-t/4C_0},t]:k_n(s) > e^{-t/4C_0} \}\big| &\leq e^{t/2C_0}\int_{t-\eps}^t k_n(s)^2 ds\\
        &\leq \frac{e^{t/2C_0}}{h_n} \sum_{j={\floor{(t-e^{-t/4C_0})/h_n}}}^{\floor{t/h_n}-1} \cD^\psi(\Ehn_{j+1}, \Ehn_j) &\text{by \cref{lem:expu}} \\
        &\leq Ce^{-t/2C_0} &\text{by \eqref{eq:dissipation exponential decay}}.
    \end{align*}
    By the above and \cref{cor:L2 control lagrange mult}, 
    \begin{equation}
        \label{eq:markov 2}
        \big|\{s\in[t-e^{-t/4C_0},t]:k_n(s) \leq e^{-t/4C_0}, |\Ehn(s)|=m\}\big| \geq e^{-t/4C_0} - Ce^{-t/2C_0} - C_th_n .
    \end{equation}
    By taking $t\geq T_1$ sufficiently large, the righthand side of \eqref{eq:markov 2} is at least $\inv{2} e^{-t/4C_0} - C_th_n$. It then follows that the righthand side is positive for $n\geq n_t$ sufficiently large, in which case \eqref{eq:11} follows.

    Next, it follows from \cref{QAT main} that if $T_1$ is sufficiently large, then $E^{(h_n)}(t_n)$ is the normal graph $f_n$ over a disjoint union $F_n$ of equally sized Wulff shapes such that $|F_n| = m$ and $\|f_n\|_{C^{1,1/2}} \leq Ce^{-t/4C_0}$. Moreover, the constraints $P_\phi(\Ehn(t_n)) \in (P_d, P_{d+1} - \delta)$ and $|P_\phi(\Ehn(t_n)) - P_\phi(F_n)|< Ce^{-t/2C_0}$ force the equality $P_\phi(F_n) = P_d$ if $T_1$ is large enough, in which case $F_n$ has exactly $d$ components and we may express $F_n = \cup_{j=1}^d W_\phi(x_{j,n},r)$ where $dr^2|W_\phi| = m$. 
    
    By \cref{lem:boundedness in finite time}, the $F_n$ are uniformly bounded, so we may pass to a subsequence in which $t_n \to t' \in [t-e^{-t/4C_0},t]$, $x_{j,n}\to x_j(t)$, and $f_n\to f_t$ in $C^{1,\aa}$ where $\aa \in (0,\inv{2})$. It follows that $F_n$ converges in Hausdorff distance to $F(t) := \cup_{j=1}^d W_\phi(x_j(t), r)$ and that the $L^1$-limit $E(t') = \lim_{n\to\infty} \Ehn(t_n)$ is the normal graph of $f_t$ over $F(t)$ such that $\|f_t\|_{C^{1,\aa}} \leq Ce^{-t/4C_0}$. In particular, we obtain the estimate $|E(t')\Delta F(t)| \leq Ce^{-t/4C_0}$.

    Recall that by \eqref{eq:exp L1 conv}, the flat flow converges exponentially fast in $L^1$ to some set $E_\infty\sb\R^2$, thus \begin{equation}
        |F(t)\Delta E_\infty| \leq |F(t)\Delta E(t')| + |E(t')\Delta E_\infty| \leq Ce^{-t/4C_0}.
    \end{equation}
    Hence $E_\infty$ is itself a disjoint union of Wulff shapes and the centers $x_j(t)$ converge exponentially fast. Thus we deduce the estimate \begin{equation}
        \label{eq:exp hausdorff conv}
        d_H(F(t), E_\infty) \leq Ce^{-t/4C_0}. 
    \end{equation}

    Since $F_n\to F(t)$ in Hausdorff distance, we find that for $n$ sufficiently large, \begin{align*}
        \sup_{\Ehn(t_n)\Delta E_\infty} d^\psi_{E_\infty} &\leq \sup_{\Ehn(t_n)\Delta F_n}d^\psi_{F_n} + d_H(F_n,E_\infty)\\
        &\leq \|f_n\|_\infty + d_H(F_n,F(t)) + d_H(F(t),E_\infty))\\
        &\leq Ce^{-t/4C_0}.
    \end{align*}
    By \cref{cor:distance comparison} and the fact that $t -t_n \leq e^{-t/4C_0} \leq (Ce^{-t/4C_0})^{1/2}$,
    \begin{equation}
        \sup_{\Ehn(t)\Delta E_\infty}d^\psi_{E_\infty} \leq Ce^{-t/36C_0}.
    \end{equation}
    Since this estimate is uniform in $n$, we obtain \eqref{eq:exp conv of flat flow}, concluding the proof of \cref{thm:exponential convergence of flat flow}:
    \begin{equation}
        \sup_{E(t)\Delta E_\infty}d^\psi_{E_\infty} \leq Ce^{-t/36C_0}.
    \end{equation}

    Finally, we observe retroactively that $\lim_{t\to\infty} p(t) = P_d$, so $\delta = \inv{2}(P_{d+1} - P_d)$ depends only on $m$ and $P_\phi(E_0)$, and hence we may remove the dependence of any prior constants on $\delta$. \hfill$\Box$

\section{Reflection Property} 
In this section, we prove \cref{thm:reflection pres by flat flow}. We recall the reflection comparison properties $(*)_H$ and $(*)_H'$ defined in \eqref{eq:reflection prop} and \eqref{eq:reflection prop'} for half-spaces $H\sb\R^N$. We occasionally use the notation $\Pi_t(\nu) := \{x\in\R^N: x\cdot \nu\leq t\}$ for half-spaces. 

\begin{lemma}
\label{lem:reflection ineq}
If $F\sb\R^N$ satisfies $(*)_H$ for some half-space $H=\Pi_\nu(s)\sb\R^N$ and $\psi\in\cM(\R^N)$ is compatible with $\nu$, then \begin{equation}
\label{eq:reflection ineq}
\sd^\psi_F(x)\leq \sd^\psi_F(\Psi(x)) \quad \forall x\in H.
\end{equation}
Moreover if $\psi$ is strictly convex and $F$ satisfies $(*)_H'$, then equality in \eqref{eq:reflection ineq} holds only if there exists $y\in \partial F\cap \partial H$ at which both distances $d^\psi_F(x)$ and $d^\psi_F(\Psi(x))$ are attained. In particular, if $F$ satisfies $(*)_H'$ and $\partial F$ is $C^1$ near $\partial F\cap \partial H$, then \eqref{eq:reflection ineq} is strict for all $x\in H$.
\end{lemma}
\begin{proof}
First let $x\in \ol{F}\setminus \text{int}(\Psi(F))$. Then \eqref{eq:reflection ineq} is trivial since $\sd^\psi_F(x) \leq 0 \leq \sd^\psi_{F}(\Psi(x))$ for such $x$. Moreover if $F$ satisfies $(*)_H'$, then equality cannot occur, since $\sd^\psi_F(x) = 0 = \sd^\psi_F(\Psi(x))$ implies $x\in (\partial F\cap \partial \Psi(F))\cap H$.

We will now prove the statement for $x\in \text{int}(\Psi(F))\cap H$ and the remaining case is a symmetric argument. Let $y \in \partial F$ be such that the minimum $\psi$-distance of $x$ to $\partial F$ is attained at $y$. If $y\in H$, then along the line segment from $x$ to $y$ there must be a point $z\in \partial \Psi(F)$, in which case by reflection symmetry of $\psi$, 
\begin{equation}
    \label{eq:reflection ineq 2}
    d^\psi_F(\Psi(x)) \leq \psi(\Psi(x) - \Psi(z)) = \psi(x-z) \leq \psi(x-y).
\end{equation} 
If $y\not\in H$, then 
\begin{equation}
    \label{eq:reflection ineq 3}
    d^\psi_F(\Psi(x)) \leq \psi(\Psi(x) - y) \leq \psi(x-y).
\end{equation}
The second inequality in \eqref{eq:reflection ineq 3} may be justified as follows: if $P$ is the hyperplane parallel to $\partial H$ which contains $y$, and $z$ is the reflection of $x$ across $P$, then $\Psi(x) = (1-t)x + tz$ for some $t\in(0,1]$ and hence
\[ \psi(\Psi(x) - y) \leq (1-t)\psi(x-y) + t \psi(z-y) = \psi(x-y).  \]
Either way, it follows that $d^\psi_F(\Psi(x)) \leq \psi(x-y) = d^\psi_F(x)$, proving \eqref{eq:reflection ineq}. 

We remark that the inequality \eqref{eq:reflection ineq 2} becomes strict if $F$ satisfies $(*)_H'$ and \eqref{eq:reflection ineq 3} is strict if $y\not\in\partial H$ and $\psi$ is strictly convex. Thus, we deduce that if $F$ satisfies $(*)_H'$ and $\psi$ is strictly convex, then equality $\sd^\psi_F(\Psi(x)) = \sd^\psi_F(x)$ can only occur if the distance $d^\psi_F(x)$ is attained at some $y\in \partial F\cap \partial H$, in which case $d^\psi_F(\Psi(x))$ is also attained at $y$ by reflection symmetry. In this case, we have $W_{\psic}(x,r)\cup W_{\psic}(\Psi(x),r)\sbq F$ where $r=\psi(x-y)$, but this cannot happen if $\partial F$ is $C^1$ near $y$.
\end{proof}

\begin{lemma}
    \label{lem:reflection stable}
    The property $(*)_H$ is stable under convergence in $L^1$. That is, if $E_\infty\sb\R^N$ is the $L^1$-limit of sets $E_j$ which satisfy $(*)_H$, then $E_\infty$ satisfies $(*)_H$.
\end{lemma}
\begin{proof}
    Suppose for sake of contradiction that $E_\infty$ does not satisfy $(*)_H$, so the set $G := (\Psi(E_\infty)\setminus E_\infty)\cap H$ is of positive measure. Since $G\sb \Psi(E_\infty)$, we see that $\Psi(E_j)\cap G\to G$ in $L^1$. Similarly, because $G\cap E_\infty = \emptyset$,  $G\setminus E_j\to G$ in $L^1$. Thus, for sufficiently large $j$, we have \[ |\Psi(E_j)\cap G| \geq \frac{2}{3}|G|, \qquad |G\setminus E_j| \geq \frac{2}{3}|G| \]
    and hence $|(\Psi(E_j)\setminus E_j)\cap G| \geq \inv{3}|G| > 0.$ As $G\sb H$, it follows that such $E_j$ do not satisfy $(*)_H$, a contradiction.
\end{proof}

In light of \cref{lem:reflection stable}, in order to prove \cref{thm:reflection pres by flat flow}, it is enough to check that $(*)_H'$ is preserved by the approximate $(\phi,\psi)$-flow.  

\begin{proposition}
    \label{prop:reflection pres by approx flow}
    Let $N=2,3$. Consider a half-space $H = \Pi_\nu(s)$ and suppose $\phi\in\cM^{2,1}(\R^N),\psi\in\cM(\R^N)$ are compatible with $\nu$. If $F\sb\R^N$ satisfies $(*)_H'$ and $\partial F$ is $C^1$ near $\partial F\cap \partial H$, then any minimizer $E$ of $\cF_h(\cdot, F)$ also satisfies $(*)_H'$. 
\end{proposition}

\begin{proof}
First we show that $E$ satisfies $(*)_H$. Suppose not for the sake of contradiction. Recalling the convention that $E$ coincides with its set of Lebesgue points, it follows that $G := (\Psi(E)\setminus E)\cap H$ is of positive measure. We claim that $\tilde{E} := (E \cup G) \setminus \Psi(G)$ satisfies $\cF_h(\tilde{E},F) < \cF_h(E, F)$, contradicting that $E$ is a minimizer. Since $|G| = |\Psi(G)|$, $G\cap E=\emptyset$, and $\Psi(G) \sb E$, we observe that $|\tilde{E}| = |E|$. 

To compare the dissipation terms, we note that $\sd^\psi_F(x) < \sd^\psi_F(\Psi(x))$ for all $x\in H$ by \cref{lem:reflection ineq}. Since $G$ is of positive measure, we have \[
    \int_G \sd^\psi_Fdx < \int_{\Psi(G)} \sd^\psi_Fdx.
\]
By adding $E\setminus \Psi(G)$ to the region of integration, it follows that $\cD^\psi(\tilde{E}, F) < \cD^\psi(E,F)$.

For the surface energies, we check $P_\phi(\tilde{E}) \leq P_\phi(E)$. Note that we can equivalently express \[ \tilde{E} = [(\Psi(E)\cup E)\cap H] \cup [\Psi(E)\cap E \cap H^c]. \]
By reflection symmetry of $\phi$ and the submodularity principle, we have \begin{align}
    \label{eq:submodularity}
    P_\phi(\Psi(E)\cup E; H) + P_\phi(\Psi(E)\cap E; \Psi(H)) &= P_\phi(\Psi(E)\cup E; H) + P_\phi(\Psi(E)\cap E; H) \\
    &\leq P_\phi(E; H) + P_\phi(\Psi(E);H)\notag \\
    &= P_\phi(E; H) + P_\phi(E; \Psi(H)).\notag
\end{align}
We further claim that $\partial^*\tilde{E}\cap \partial H$ and $\partial^* E\cap \partial H$ are equivalent up to a $\cH^{N-1}$-null set, from which the desired inequality $P_\phi(\tilde{E})\leq P_\phi(E)$ then immediately follows after \eqref{eq:submodularity} and \cite[Theorem 16.3]{maggi2012}.
It suffices to show that if $x\in \partial H$, then $x\in \tilde{E}^{(1/2)}$ if and only if $x\in E^{(1/2)}$. This is true since for any $r>0$, \begin{align*}
    |\tilde{E}\cap B_r(x)| &= |(E\cup \Psi(E))\cap H\cap B_r(x)| + |E\cap \Psi(E) \cap H^c \cap B_r(x)|\\
    &= |(E\cup \Psi(E))\cap H\cap B_r(x)| + |E\cap \Psi(E) \cap H \cap B_r(x)|\\
    &= |E\cap H \cap B_r(x)| + |\Psi(E)\cap H\cap B_r(x)|\\
    &= |E\cap H\cap B_r(x)| + |E\cap H^c \cap B_r(x)|\\
    &= |E\cap B_r(x)|. 
\end{align*}
Thus we have shown $E$ satisfies $(*)_H$.

Lastly we show that $E$ satisfies $(*)_H'$. Suppose for the sake of contradiction that $\partial E$ and $\partial \Psi(E)$ intersect at some point $x\in H$. Recall that since $N\leq 3$, $E$ is $C^{2,\aa}$ by \cref{prop:elliptic reg}, so by the monotonicity principle \eqref{curvature monotonicity}, we have $\kappa^\phi_E(\Psi(x)) = \kappa^\phi_{\Psi(E)}(x) \geq \kappa^\phi_E(x)$, and hence $\sd^\psi_F(x) \geq \sd^\psi_F(\Psi(x))$ by the Euler-Lagrange equation \eqref{eq:euler-lagrange}. However, this contradicts $\sd^\psi_F(x) < \sd^\psi_F(\Psi(x))$ from \cref{lem:reflection ineq}, completing the proof.
\end{proof}

\noindent \textbf{Proof of \cref{thm:reflection pres by flat flow}:}
\medskip

Since $N\leq 3$, recall the approximate flow $\Eh(t)$ is $C^{2,\aa}$ by \cref{prop:elliptic reg} for all $t\geq h$.
Thus by iterating \cref{prop:reflection pres by approx flow}, if the initial set $E_0$ satisfies $(*)_H'$, then $\Eh(t)$ satisfies $(*)_H'$ for all $t\geq0$. By \cref{lem:reflection stable}, the flat flow $E(t)$ also satisfies $(*)_H$. \hfill $\Box$

\begin{remark}
    The restriction to dimensions $N\in\{2,3\}$ in \cref{thm:reflection pres by flat flow} is needed so that the Euler-Lagrange equation \eqref{eq:euler-lagrange} holds globally on the approximate flow.
\end{remark}

\subsection{On the Long-term Profile}
\label{sec:reflection long-term}

We now apply \cref{thm:exponential convergence of flat flow} in tandem with \cref{thm:reflection pres by flat flow} to prove \cref{cor:diameter bound} and \cref{cor:single wulff}, which each establish some conditions on the initial set that allow us to confine the long-term profile of the flat flow in the plane.

\begin{proof}[Proof of \cref{cor:diameter bound}]
    Fix $\eps>0$ and let $D_\eps = (1+\eps)D$. For every half-space $H$ containing $D_\eps$, $E_0$ is contained in $H$ and the intersection $\partial E_0\cap \partial H$ is empty. Thus $E_0$ trivially satisfies $(**)_{D_\eps, \cP}'$ as well as the condition that $E_0$ is $C^1$ near $\partial H$ for all such half-spaces. Hence $E_\infty$ satisfies $(**)_{D_\eps, \cP}$ by \cref{thm:reflection pres by flat flow}, so $E_\infty$ must be contained in $D_\eps + (|E_0|/|W_\phi|)^{1/2} W_\phi$ since $D_\eps$ is convex. As $\eps$ is arbitrary, the claim follows.
\end{proof}

\begin{proof}[Proof of \cref{cor:single wulff}]

    Recall that by \cref{thm:exponential convergence of flat flow}, the limiting set $E_\infty$ is a disjoint union of Wulff shapes $\cup_{j=1}^d W_\phi(x_j, r)$ where $r = (\frac{|E_0|}{d|W_\phi|})^{1/2}$. By \cref{thm:reflection pres by flat flow}, $E_\infty$ satisfies $(**)_{D,\cP}$, so the centers $x_j$ must lie in $\ol{D}$; otherwise, the property $(*)_H$ would be violated for some half-space $H$ which contains $D$ but excludes $x_j$. In particular, $E_\infty \sb D + rW_\phi$ so $|D + rW_\phi| \geq |E_0|$. Since $r$ is decreasing with respect to $d$, it suffices to show $|D+rW_\phi| < |E_0|$ when $d=2$. 
        
    As $D$ and $W_\phi$ are convex, it follows by a standard result of mixed volumes that \[ |D+rW_\phi| = |D| + rP_\phi(D) + r^2|W_\phi|. \]
    Thus it is equivalent to show \[ |D| + \aa\sqrt{2|D||E_0|} - \frac{|E_0|}{2} < 0 \] which occurs precisely when
    \begin{equation*}
        \ps{\frac{|D|}{|E_0|}}^{1/2} < \frac{\sqrt{\aa^2+1} - \aa}{\sqrt{2}}.
        \qedhere
    \end{equation*}
\end{proof}

\subsection{Stability}
\label{sec:reflection stability}
In this section, we prove \cref{thm:reflection pres by flat flow stability}.

We say that $\cP$ is a \textit{root system} if \begin{equation}
    \label{eq:root system}
    \forall \nu\in\cP, \quad -\nu\in\cP\text{ and }\Psi_\nu(\cP) = \cP.
\end{equation}

Moreover, we say that $E\sb\R^N$ satisfies the \textit{$r$-cone property} (with respect to $\cP$) if for all $x\in\R^N$, there exists a basis $A\sb\cP$ of $\R^N$ such that 
\newcommand{\cone}{\operatorname{cone}}
    \begin{align*}
        x + r\cone(A) &\sbq E^c & &\text{if }x \in E^c\\
        x - r\cone(A) &\sbq E & &\text{if }x\in E
    \end{align*}
where $\cone(A)$ is defined to be the open convex hull of $A\cup \set{0}$. 
  
It is shown in \cite{kim2021volume} that a set which satisfies $(**)_{D,\cP}$ for a large enough family of half-spaces satisfies a uniform cone condition and hence enjoys uniform Lipschitz regularity: 

\begin{proposition}{\cite[Theorem 2.2]{kim2021volume}}
    \label{prop:reflection lipschitz}
    Suppose $\cP\sb\S^{N-1}$ is a root system such that $\text{span}(\cP\setminus K) = \R^N$ for any hyperplane $K$ through the origin. Then there exists $c=c(\cP)$ such that the following holds: if $E\sb\R^N$ satisfies $(**)_{B_\rho(0),\cP}$ for some $\rho < c|E|^{1/N}$, then $E$ satisfies the $r$-cone condition for some $r>0$, with locally constant cone directions which are independent of $E$. 
\end{proposition}
\begin{remark}
    We refer the reader to \cite{kim2021volume} for a precise characterization of the constant $c(\cP)$. See also Appendix~\ref{ap:ref} for the characterization in $\mathbb{R}^2$.
\end{remark}

\begin{lemma}[Compactness result]
    \label{lem:reflection compactness}
    Let $\cP\sb\S^{N-1}$ be a root system, and suppose $E_k\sb \R^N$ is a uniformly bounded sequence of sets satisfying the statement of \cref{prop:reflection lipschitz} for $\cP$. Then, there exists a subsequence that converges in Hausdorff distance and in $L^1$ (to the same limit).
\end{lemma}
\begin{proof}
    Suppose $E_k \sb B_R$ for some $R>0$. We observe that by \cref{prop:reflection lipschitz}, there exists a finite open cover $(U_i)_i$ of $B_R$ with corresponding normal vectors $\nu_i\in \S^{N-1}$ such that for all $k$ and $i$, $\partial E_k\cap U_i$ either is empty or is given by a Lipschitz graph 
    \begin{equation}
        \label{eq:lipschitz graph}
        x\cdot \nu_i = f^{(i)}_k(x - (x\cdot \nu_i)\nu_i)
    \end{equation}
    where the Lipschitz constant of $f^{(i)}_k$ is uniform with respect to $i,k$. 

    Passing to a subsequence, we can ensure that for each $i$, one of the following is true: \begin{enumerate}
        \item $U_i \sb E_k$ for all $k$
        \item $U_i \sb E_k^c$ for all $k$
        \item $\partial E_k\cap U_i$ is a Lipschitz graph over $U_i$ of the form \eqref{eq:lipschitz graph} for all $k$.
    \end{enumerate}
    The desired statement now follows from Arzel\`{a}-Ascoli.
\end{proof}

\begin{lemma}[$\Gamma$-convergence]
    \label{lem:gamma conv}
    Suppose $(F_k)_{k\geq1}, (E_k)_{k\geq1}$ are sequences of bounded sets of finite perimeter in $\R^N$ such that $E_k\to E$ in $L^1$ and $F_k\to F$ in Hausdorff distance. Then \begin{equation}
        \label{eq:gamma conv}
        \cF_h(E,F) \leq \liminf_{k\to\infty} \cF_h(E_k,F_k).
    \end{equation}
    In particular, if $E_k \in \argmin \cF_h(\cdot, F_k)$ for all $k$, then $E\in \argmin \cF_h(\cdot, F)$. 
\end{lemma}
\begin{proof}
    Note that $L^1$-convergence $E_k\to E$ implies $|E_k|\to |E|$ and $P_\phi(E) \leq \liminf_{k\to\infty} P_\phi(E_k)$ by lower-semicontinuity. Moreover, by \cref{lem:standard estimate flat flow}(i) and the Hausdorff convergence $F_k\to F$, we observe that the $E_k$ are uniformly bounded and that $\sd^\psi_{F_k}$ converges uniformly to $\sd^\psi_F$. Thus for the dissipation terms, we may estimate \begin{align*}
        \abs{\int_E \sd^\psi_Fdx - \int_{E_k} \sd^\psi_{F_k}dx} &\leq \int_{E\Delta E_k} d^\psi_F + \int_{E_k} |\sd^\psi_F - \sd^\psi_{F_k}|dx \xrightarrow{k\to\infty} 0,
    \end{align*}
    proving \eqref{eq:gamma conv}. 

    Now assume $E_k \in \argmin \cF_h(\cdot, F_k)$ and suppose there exists $G$ such that $\cF_h(G,F) < \cF_h(E,F)$. By the previous argument, we find \[ \lim_{k\to\infty} \cF_h(G,F_k) = \cF_h(G,F) < \cF_h(E,F) \leq \liminf_{k\to\infty} \cF_h(E_k, F_k) \] 
    which for large enough $k$ contradicts the fact that $E_k$ minimizes $\cF_h(\cdot, F_k)$. 
\end{proof}

\begin{lemma}
    \label{lem:reflection approx}
    Suppose $E\sb\R^N$ is bounded and satisfies $(**)_{B_\rho(0),\cP}$, where $\cP\sb\S^{N-1}$ is a root system and $\rho < c(\cP)$ is such that \cref{prop:reflection lipschitz} applies. Then for all $\eps>0$, there exists a $C^\infty$ set $F$ which satisfies $(**)_{B_\rho(0),\cP}'$ and $d_H(E,F) < \eps$. 
\end{lemma}

\begin{proof}
    Fix $\eps>0$, and consider the set \[ F:= \set{x\in E: d_E(x) > \eps g(x) }\qquad \text{where}\qquad g(x) := \inv{2} \sum_{\nu\in\cP} |x\cdot \nu| = \sum_{\nu\in \cP:\,x\cdot\nu > 0} x\cdot\nu. \]
    We claim that $F$ satisfies $(**)_{B_\rho(0),\cP}'$. Fix a half-space $H= \Pi_{\nu_0}(s)$ containing $B_\rho(0)$ where $\nu_0\in\cP$ and the reflection map $\Psi(x) = \Psi_{H}(x) = x + 2(s-x\cdot\nu_0)\nu_0$. We wish to prove
    \begin{equation}
        \label{eq:reflection ineq approx}
        d_E(\Psi(x)) - \eps g(\Psi(x)) > d_E(x) - \eps g(x) \qquad\forall x \in E\cap \Psi(H),
    \end{equation}
    since it then follows that $\Psi(x) \in \text{int}(F)$ for all $x\in\ol{F}\cap \Psi(H)$, implying $(*)_H'$. By \cref{lem:reflection ineq} we have $d_E(\Psi(x)) \geq d_E(x)$, so it suffices to show 
    \begin{equation}
        \label{eq:g reflection ineq}
        g(\Psi(x)) < g(x)\qquad\forall x\in \Psi(H).
    \end{equation}

    First, we show the inequality 
    \begin{equation}
        \label{eq:g ineq}
        y\cdot \nu_0\geq 0, t\geq 0 \qquad\implies\qquad  g(y + t\nu_0) - g(y) \geq t.
    \end{equation}
    Indeed, using the fact that $\Psi_{\nu_0}(\cP)=\cP$, we may estimate the directional derivative $D_{\nu_0}g(y)$ (where it exists) for any $y$ such that $y\cdot\nu_0 > 0$ as follows: 
    \begin{align}
        \label{eq:g dir deriv}
        D_{\nu_0} g(y) &= \sum_{\nu\in\cP} (\nu\cdot\nu_0) 1_{\{y\cdot\nu > 0\}} \\
        &= \sum_{\nu\in\cP:\, \nu\cdot\nu_0 > 0} \bks{(\nu\cdot\nu_0)1_{\{y\cdot\nu>0\}} + (\Psi_{\nu_0}(\nu)\cdot\nu_0)1_{\{y\cdot\Psi_{\nu_0}(\nu)>0\}}} \notag\\
        &= \sum_{\nu\in\cP:\, \nu\cdot\nu_0>0} (\nu\cdot\nu_0)\bks{1_{\{y\cdot\nu>0\}} - 1_{\{y\cdot\Psi_{\nu_0}(\nu)>0\}}}\notag \\
        &\geq \nu_0 \cdot \nu_0 = 1\notag
    \end{align}
    where the last inequality follows from observing that if $\nu\cdot\nu_0 > 0$ and $y\cdot\Psi_{\nu_0}(\nu)>0$, then \[ y\cdot\nu = y\cdot \Psi_{\nu_0}(\nu) + 2(\nu\cdot\nu_0)(y\cdot\nu_0) > 0. \]
    Since $D_{\nu_0}g$ exists at cofinitely many points along any line parallel to $\nu_0$, \eqref{eq:g ineq} follows from integrating \eqref{eq:g dir deriv}. Applying \eqref{eq:g ineq}, we deduce that for any $x \in \Psi(H)$, 
    \begin{align*}
        g(x) - g(\Psi(x)) &\geq 2(x\cdot\nu_0 - s) &\text{when } x\cdot\nu_0 \in (s, 2s)\\
        g(x) - g(\Psi(x)) &= g(x) - g(\Psi_{\nu_0}(\Psi(x))) \geq 2s &\text{when } x\cdot\nu_0 \in [2s,\infty)
    \end{align*}
    which implies \eqref{eq:g reflection ineq}, completing the proof that $F$ satisfies $(**)_{B_\rho(0),\cP}'$. 

    Note we must further modify $F$ to be $C^\infty$, for which we use a radial mollifier. Let $\rho\in C_c^\infty(\R^N)$ be such that $\rho\geq0$, $\spt(\rho) = B_1(0)$, $\int \rho(x)dx=1$, and $\rho(x) = \eta(|x|)$ where $\eta$ is strictly decreasing on its support.  Consider $u_\eps := 1_F * \rho_\eps$ where $\rho_\eps = \eps^{-N}\rho(\cdot/\eps)$. By Sard's theorem, the set $F_\eps := \set{u_\eps > t}$ is $C^\infty$ for some $t \in (0,1)$. 
    
    We claim $F_\eps$ also satisfies $(**)_{B_\rho(0),\cP}'$. Let $H$ be a half-space such that $F$ satisfies $(*)_H'$. We will show $F_\eps$ also satisfies $(*)_H'$, for which it suffices to show that for any $x\in \Psi(H)$, $u_\eps(x) \geq t$ implies $u_\eps(\Psi(x)) > t$. For $x\in\Psi(H)$, we have
    \begin{equation}
        u_\eps(\Psi(x)) - u_\eps(x) = \int_F [\rho_\eps(z - \Psi(x)) - \rho_\eps(z - x)] dz
    \end{equation}
    Observing the identity \begin{equation*}
        \int_{F\cap \Psi(F)} \rho_\eps(z - \Psi(x)) dz= \int_{F\cap \Psi(F)} \rho_\eps(z - x) dz
    \end{equation*}
    due to reflection symmetry, we may further express 
    \begin{equation}
        \label{eq:mollifier reflection}
        u_\eps(\Psi(x)) - u_\eps(x) 
        = \int_{F\setminus \Psi(F)} [\rho_\eps(z- \Psi(x)) - \rho_\eps(z - x)]dz.
    \end{equation}
    Note that $F\setminus\Psi(F)\sb H$ since $F$ satisfies $(*)_H$, and $|z - \Psi(x)| < |z-x|$ for any $z\in H$. Thus from \eqref{eq:mollifier reflection} we deduce $u_\eps(\Psi(x)) \geq u_\eps(x)$. 
    
    It remains to verify $u_\eps(\Psi(x)) > t$ in the case that $u_\eps(x) = t$. Suppose we have $u_\eps(\Psi(x)) = u_\eps(x)=t$. Then \eqref{eq:mollifier reflection} as well as the strict monotonicity of $\eta$ implies that $B_\eps(\Psi(x)) \cap [F\setminus \Psi(F)]$ is empty. Since $F$ satisfies $(*)_H'$, the space $(\Psi(F) \cup F^c)\cap H$ is disconnected, so we may deduce $B_\eps(\Psi(x)) \cap H$ is contained in either $\Psi(F)$ or $F^c$. Then $B_\eps(\Psi(x))$ is contained in either $\Psi(F)$ or $F^c$, so $u_\eps(\Psi(x)) \in \set{0,1}$, which is a contradiction.  Thus we deduce $F_\eps$ satisfies $(**)_{B_\rho(0),\cP}'$. 

    Lastly, we need to check that $F_\eps$ is close to $E$ with respect to Hausdorff distance. Since $F\sbq E$, we have $d_H(E,F) = \sup_{x\in E} \dist(x,F)$. Recall that by \cref{prop:reflection lipschitz}, $E$ satisfies the $r$-cone property with respect to $\cP$ for some $r>0$. For a fixed $x\in E$, let $A\sb\cP$ be a basis such that $x - r\cone(A)\sb E$. Then for $\eps \lesssim_\cP r$ sufficiently small, we may find a point $y\in x - r\cone(A)$ such that $d_E(y) > \eps \sup_{z\in E} g(z) \geq \eps g(y)$ and $|x-y| \sim_\cP \eps$. In particular, $y$ belongs to $F$, so $d_H(E,F) \lesssim_\cP \eps$.  

    Similarly we find that $d_H(F,F_\eps) \lesssim_\cP \eps$. The bound $\sup_{x\in F_\eps} \dist(x,F) \leq \eps$ is immediate from definition, while the bound $\sup_{x\in F} \dist(x,F_\eps) \lesssim_\cP \eps$ can be argued as above using the interior $r$-cone property of $F$.
\end{proof}

\noindent \textbf{Proof of \cref{thm:reflection pres by flat flow stability}:}
\medskip

By \cref{lem:reflection approx}, there exists a sequence of $C^1$ sets $E_k\sb\R^N$ which satisfy $(**)_{B_\rho(0),\cP}'$, such that $E_k\to E_0$ in Hausdorff distance. Fix $h>0$, and let $\Eh_k(t)$ be an approximate flow with initial set $E_k$. By applying \cref{lem:reflection compactness} and \cref{lem:boundedness in finite time} with a diagonalization argument, we may pass to a subsequence such that for all $t\geq0$, $\Eh_k(t)$ converges in $L^1$ and in Hausdorff distance to some set $\tilde{E}^{(h)}(t)$ as $k\to\infty$. Note that by \cref{lem:gamma conv}, $\tilde{E}^{(h)}(t)$ is an approximate flow starting from $E$. Moreover $\Eh_k(t)$ satisfies $(**)_{B_\rho(0),\cP}'$ for all $t\geq0$, and hence $\tilde{E}^{(h)}(t)$ satisfies $(**)_{B_\rho(0),\cP}$ by \cref{lem:reflection stable}. By diagonalizing w.r.t $h$ and applying \cref{lem:reflection stable} once more, we obtain a flat flow from $E_0$ which satisfies $(**)_{B_\rho(0),\cP}$, and we are done. \hfill $\Box$

\begin{corollary}
    Under the same setting as in \cref{thm:reflection pres by flat flow stability}, we additionally assume $N=2$. Then there exists an area-preserving flat $(\phi,\psi)$-flow $E(t)$ starting from $E_0$ which converges to a single Wulff shape.
\end{corollary}

\begin{proof}
    By \cref{thm:reflection pres by flat flow stability}, there exists a flat $(\phi,\psi)$-flow $E(t)$ starting at $E_0$ which satisfies $(**)_{B_\rho(0),\cP}$ for all $t\geq0$. Thanks to \cref{prop:b}, $E(t)$ is a star-shaped domain for all $t \geq 0$
    and thus, the limit given in \cref{thm:exponential convergence of flat flow} is a single Wulff shape.  
\end{proof}

\appendix

\section{Proof of Proposition~\ref{QAT lemma}}
\label{ap:proof3}

For any norm $\phi$, we recall the anisotropic isoperimetric inequality \cite{maggi2012}: for any $E\sb\R^N$, \begin{equation}
    \label{isoperimetric inequality}
    P_\phi(E) \geq N |W_\phi|^{1/N} |E|^{(N-1)/N}.
\end{equation}

\begin{proof}[Proof of Proposition~\ref{QAT lemma}]
Our argument makes frequent use of \cref{anisotropic gauss bonnet}, so it is more natural to work with the measure $dP_\phi$ rather than $d\cH^1$. Hence, we consider the $\phi$-weighted average \[ \tilde{\kappa}_E^\phi = \inv{P_\phi(E)} \int_{\partial E} \kappa^\phi_E dP_\phi. \]
Note that \[ \|\kappa_E^\phi - \ol{\kappa}^\phi_E\|_{L^2(\partial E)} \leq  \|\kappa_E^\phi - \tilde{\kappa}^\phi_E\|_{L^2(\partial E)} \leq L_\phi\|\kappa_E^\phi - \tilde{\kappa}^\phi_E\|_{L^2(dP_\phi)} \]
and similarly $\|\kappa_E^\phi - \ol{\kappa}^\phi_E\|_{L^2(\partial E)} \geq L_\phi^{-1}\|\kappa_E^\phi - \tilde{\kappa}^\phi_E\|_{L^2(dP_\phi)}$. Thus, it makes no harm to assume $\|\kappa_E^\phi - \tilde{\kappa}_E^\phi\|_{L^2(dP_\phi)} \leq \eps_0$ and to show the desired bounds in (a) for $\tilde{\kappa}_E^\phi$. 

Let $E_1, \dots, E_d$ be the connected components of $E$, $\Gamma_i$ the outer component of $\partial E_i$, and $\hat{E_i}$ the interior of $\Gamma_i$. Let $c,C$ be constants depending only on $m,M,L_\phi$, which may change from line to line. In particular, they may depend on $|W_\phi|$ since $\pi L_\phi^{-2} \leq |W_\phi| \leq \pi L_\phi^2$. We proceed in the following steps:
\begin{enumerate}
    \item \ul{There exists $\hat{E}_k$ such that $|\hat{E}_k|\geq c$}: 

    Letting $Q = (0,1)^2$, we claim there is a constant $\tilde{c} = \tilde{c}(\phi)>0$ such that 
    \begin{equation*}
        \sup_{z\in\Z^2} |E\cap (z+Q)| \geq \tilde{c}\frac{|E|^2}{P_\phi(E)^2} \geq c
    \end{equation*}
    Indeed, letting $\bb := \sup_{z\in\Z^2} |E\cap (z+Q)|$, we have by the relative isoperimetric inequality, \begin{align*}
        P_\phi(E) &= \sum_{z\in \Z^2} P_\phi(E; z+Q) \geq \tilde{c} \sum_{z\in\Z^2} |E\cap (z+Q)|^{1/2} \geq \frac{\tilde{c}}{\bb^{1/2}}\sum_{z\in\Z^2} |E\cap (z+Q)| = \frac{\tilde{c}}{\bb^{1/2}}|E|.
    \end{align*}
    Without loss of generality, we may translate the components $E_i$ to satisfy $\text{dist}(E_i, E_j) > \sqrt{2}$ for $i\neq j$, in which case each square $z+Q$ intersects only one component $E_i$. Thus there exists $k$ such that $|\hat{E}_k| \geq |E_k| \geq c$.

    \item \ul{$c\leq |\tilde{\kappa}_E^\phi| \leq C$}: 

    For any boundary component $\Gamma$ of $\partial E$, we can use \cref{anisotropic gauss bonnet} to bound
    \begin{align}
        \label{QAT step 2}
        \abs{\tilde{\kappa}_E^\phi - \frac{2|W_\phi|}{P_\phi(\Gamma)}} &= |\tilde{\kappa}_E^\phi - \tilde{\kappa}_{\Gamma}^\phi|= P_\phi(\Gamma)^{-1/2} \|\tilde{\kappa}_E^\phi - \tilde{\kappa}_{\Gamma}^\phi\|_{L^2(\Gamma, dP_\phi)}\\
        &\leq P_\phi(\Gamma)^{-1/2}(\|\kappa_E^\phi - \tilde{\kappa}_{E}^\phi\|_{L^2(\partial E, dP_\phi)} + \|\kappa_E^\phi - \tilde{\kappa}_{\Gamma}^\phi\|_{L^2(\Gamma, dP_\phi)}) \notag \\
        &\leq 2\eps_0 P_\phi(\Gamma)^{-1/2} \notag.
    \end{align}
    For $\eps_0 < \frac{|W_\phi|}{2M^{1/2}}$, we obtain the lower bound \[ 
        \tilde{\kappa}_E^\phi \geq \frac{2(|W_\phi| - \eps_0M^{1/2})}{P_\phi(\Gamma)} \geq \frac{|W_\phi|}{M}. 
    \]
    For an upper bound, note that $P_\phi(\Gamma_k) \geq 2|W_\phi|^{1/2} |\hat{E}_k|^{1/2} \geq c$ by the anisotropic isoperimetric inequality \eqref{isoperimetric inequality} and step 1. Hence applying \eqref{QAT step 2} to $\Gamma_k$ yields \[ \tilde{\kappa}_E^\phi \leq \frac{2(|W_\phi| + \eps_0 M^{1/2})}{P_\phi(\Gamma_k)} \leq \frac{3|W_\phi|}{c}. \]

    \item \ul{$P_\phi(\Gamma_j) \geq c$ for all $j$}:
    For any $j$, we may use \cref{anisotropic gauss bonnet} and Cauchy-Schwarz to bound
    \begin{align*}
        2|W_\phi| &= \int_{\Gamma_j} \kappa_E^\phi dP_\phi\\
        &\leq P_\phi(\Gamma_j) |\tilde{\kappa}_E^\phi| + \int_{\Gamma_j} |\kappa_E^\phi - \tilde{\kappa}_E^\phi| dP_\phi\\
        &\leq P_\phi(\Gamma_j) |\tilde{\kappa}_E^\phi| + P_\phi(\Gamma_j)^{1/2} \|\kappa_E^\phi - \tilde{\kappa}_E^\phi\|_{L^2(\partial E, dP_\phi)}\\
        &\leq P_\phi(\Gamma_j)^{1/2}(M^{1/2} C + \eps_0).
    \end{align*}
    
    Thus we deduce $P_\phi(\Gamma_j)\geq c$ and hence $E$ has at most $M/c$ components. 
    
    \item \ul{For $\eps_0$ small enough, each component $E_j$ is simply connected}:
    
    If there is some component of $E$ which is not simply connected, then we may find two boundary components $\Gamma_i$, $\Gamma_j \sb \partial E$ such that \[ \int_{\Gamma_i} \kappa_E^\phi d P_\phi = 2|W_\phi|, \qquad \int_{\Gamma_j} \kappa_E^\phi d P_\phi = -2|W_\phi|. \]
    Hence by the triangle inequality and \eqref{QAT step 2}, we obtain the bound \[ \frac{4|W_\phi|}{M} \leq \frac{2|W_\phi|}{P_\phi(\Gamma_i)} + \frac{2|W_\phi|}{P_\phi(\Gamma_j)} \leq \abs{\tilde{\kappa}_E^\phi - \frac{2|W_\phi|}{P_\phi(\Gamma_i)} } + \abs{\tilde{\kappa}_E^\phi + \frac{2|W_\phi|}{P_\phi(\Gamma_j)} } \leq 2\eps_0 (P_\phi(\Gamma_i)^{-1/2} + P_\phi(\Gamma_j)^{-1/2}) \leq \frac{4\eps_0}{c^{1/2}}.  \]
    For $\eps_0$ small enough, the above inequality is a contradiction, in which case all components of $E$ must be simply connected.  \qedhere
\end{enumerate}
\end{proof}

\section{Flat Flow Standard Estimates}
\label{ap:flat}

The following lemmas are a collection of results adapted from \cite{mugnai2016}. We recall the functional \[ \cF_h(E,F) := P_\phi(E) + \inv{h} \int_{E} \sd^\psi_F(x)dx + \inv{\sqrt{h}} ||F|-m| \]
and assume throughout this section that $\phi,\psi$ are norms on $\R^N$.

Before showing the standard estimates, we need a one-sided density lemma:

\begin{lemma}
    \label{one-sided density lemma}
    Let $\phi$ be a norm on $\R^N$. There is a constant $c = c(N,L_\phi)$ such that the following holds: Suppose $E\sb\R^N$ is a bounded set of finite perimeter and $x\not\in E$, $r_0>0$ are such that \begin{equation}
        \label{eq:one-sided min}
        P_\phi(E) \leq P_\phi(E\cup B_r(x)) \qquad \forall\, 0<r<r_0.
    \end{equation}
    Then $|B_r(x)\setminus E| \geq c r^N$ for all $r < r_0$. 
\end{lemma}
\begin{proof}
    Define $f(r) := |B_r(x)\setminus E|$. Recalling the convention that $E$ is its Lebesgue representative, we have $f(r)>0$ for all $r>0$ since $x\not\in E$. Note that by monotonicity and the coarea formula, for a.e. $r>0$ we have $f$ is differentiable and $f'(r) = \cH^{N-1}(\partial B_r(x)\setminus E)$. 
    
    Moreover, for a.e. $r>0$ we have $\cH^{N-1}(\partial^*E \cap \partial B_r(x))=0$, yielding the following identities via \cite[Theroem 16.3]{maggi2012}: \begin{align*}
        P_\phi(B_r(x)\setminus E) &= P_\phi(B_r(x) ; E^c) + P_\phi(E;B_r(x))\\
        P_\phi(E\cup B_r(x)) &= P_\phi(B_r(x); E^c) + P_\phi(E;B_r(x)^c).
    \end{align*}
    Adding the two identities and invoking \eqref{eq:one-sided min}, we obtain \begin{align*}
        2P_\phi(B_r(x);E^c) &= P_\phi(B_r(x)\setminus E) + P_\phi(E\cup B_r(x)) - P_\phi(E)\\
        &\geq P_\phi(B_r(x)\setminus E). 
    \end{align*}
    By the anisotropic isoperimetric inequality \eqref{isoperimetric inequality}, we have the bound
    \begin{align*}
        \cH^{N-1}(B_r(x)\setminus E) \gtrsim_{L_\phi} 2P_\phi(B_r(x);E^c) \geq P_\phi(B_r(x)\setminus E) \gtrsim_{N,L_\phi} |B_r(x)\setminus E|^{(N-1)/N}.
    \end{align*}
    That is, we have shown $f'(r) \gtrsim_{N,L_\phi} f(r)^{(N-1)/N}$, which implies $f(r) \gtrsim_{N,L_\phi} r^N$ upon integrating. 
\end{proof}

\begin{lemma}[$L^\infty$ estimate, \cref{lem:standard estimate flat flow}(i)]
    \label{Linfty estimate}
    Let $F\sb\R^N$ be a bounded set of finite perimeter and let $E$ be a minimizer of $\cF_h(\cdot, F)$. Then there exists a constant $c=c(N,L_\phi,L_\psi)>0$ such that $\sup_{E\Delta F} d^\psi_F \leq c\sqrt{h}$.
\end{lemma}
\begin{proof}
    In what follows, $C$ is a constant depending only on $N,L_\phi$, which may change from line to line. Suppose $x_0\in E\Delta F$ satisfies $d^\psi_F(x_0) > c\sqrt{h}$ for some $c>2$. 

    Without loss of generality, assume $x_0\in F\setminus E$ and hence $\sd^\psi_F(x_0) < -c\sqrt{h}$, for the other case is symmetric. Then for any $r \leq \frac{c\sqrt{h}}{L_\psi}$, we have $B_r(x_0) \sb F$, and thus from the inequality $\cF_h(E,F)\leq \cF_h(E\cup B_r(x_0), F)$, it follows \begin{equation}
        P_\phi(E)\leq P_\phi(E\cup B_r(x_0)) + \inv{h} \int_{B_r(x_0)\setminus E} \sd^\psi_F dx + \inv{\sqrt{h}}|B_r(x_0) \setminus E|.
    \end{equation}
    For $r \leq \frac{c\sqrt{h}}{2L_\psi}$, by the triangle inequality we have that $\sd^\psi_F < -\frac{c\sqrt{h}}{2}$ on $B_r(x_0)$, and hence \begin{equation}
        P_\phi(E)\leq P_\phi(E\cup B_r(x_0)) - \inv{\sqrt{h}}\ps{\frac{c}{2} - 1}|B_r(x_0) \setminus E|.
    \end{equation}
    Thus \cref{one-sided density lemma} applies for any $\Lambda>0$, and hence $|B_r(x_0)\setminus E| \geq C^{-1}r^N$. Since \begin{align*}
        \inv{\sqrt{h}}\ps{\frac{c}{2} - 1}|B_r(x_0) \setminus E| &\leq P_\phi(E\cup B_r(x_0)) - P_\phi(E)\\
        &\leq \int_{\partial B_r(x_0)\setminus E} dP_\phi\\
        &\leq L_\phi \cH^{N-1}(\partial B_r(x_0)\setminus E) \leq C r^{N-1},
    \end{align*}
    it follows that \begin{equation}
        r \leq \ps{\frac{c}{2}-1}^{-1} C\sqrt{h}.
    \end{equation}
    Taking $r = \frac{c\sqrt{h}}{2L_\psi}$ yields \[ c \leq \ps{\frac{c}{2} -1}^{-1}2CL_\psi \]
    which implies $c$ is bounded above by some function of $N,L_\phi,L_\psi$.
\end{proof}

We observe the following standard density estimate for $(\Lambda,r_0)$-minimizers, which is proven, for instance, in \cite[Lemma 2.8]{philippis2014}.

\begin{lemma}
    \label{density lemma for almost minimizers}
Let $\phi$ be a norm on $\R^N$ and $E\sb\R^N$ a $(\Lambda, r_0)$-minimizer of $P_\phi$. There exists a constant $C = C(N,L_\phi)>0$ such that for all $x\in\partial E$ and $r<\min\{r_0,\Lambda^{-1}\}$, 
\begin{align}
    \label{almost perimeter minimality 1}
    \min\{|E\cap B_r(x)|, |B_r(x)\setminus E|\}&\geq C^{-1}r^N\\
    \label{almost perimeter minimality 2}
    C^{-1} r^{N-1} \leq P_\phi(E;B_r(x)) &\leq Cr^{N-1}.
\end{align} 
\end{lemma}
\begin{remark}
    Now that we have shown the $L^\infty$ estimate, we recall by \cref{flat flow is almost minimizer} that $\cF_h$-minimizers are $(\Lambda, r_0)$-minimizers for any $r_0>0$ and $\Lambda = \frac{c+1}{\sqrt{h}} + \frac{2L_\psi r_0}{h}$, where $c$ is the constant from \cref{Linfty estimate}. In particular, $r_0$ can be chosen so that $\Lambda r_0 = 1$ and $r_0 \sim_{N,L_\phi,L_\psi} \sqrt{h}$, so the estimates in \cref{density lemma for almost minimizers} are valid for all $r\leq c \sqrt{h}$, up to modifying the constant $C$. This is useful for proving the $L^1$ and $L^2$ estimates.
\end{remark}

\begin{lemma}[$L^1$ estimate, \cref{lem:standard estimate flat flow}(ii)]
    \label{lem:L1 estimate}
    Let $F\sb\R^N$ be a bounded set of finite perimeter and let $E$ be a minimizer of $\cF_h(\cdot, F)$. Then there exist a constant $C=C(N,L_\phi,L_\psi)>0$ such that for all $\ell\leq c\sqrt{h}$, \[ |E\Delta F| \leq C\ps{\ell P_\phi(E) + \inv{\ell} \cD^\psi(E,F) }. \]
\end{lemma}
\begin{proof}
    We split up $E\Delta F$ into two regions:
    \[ |E\Delta F| \leq |\{x\in E\Delta F: d_F(x) \leq \ell\}| + |\{x\in E\Delta F: d_F(x) \geq \ell\}|. \]
    The second term is easily bounded via Markov's inequality:
    \[|\{x\in E\Delta F: d_F(x) \geq \ell\}| \leq \frac{L_\psi}{\ell}\int_{E\Delta F} d^\psi_F(x)dx = \frac{L_\psi}{\ell} \cD^\psi(E,F). \]
    For the first term, we apply a covering argument. By Vitali's covering lemma, we may find a finite disjoint collection of balls $\{B_\ell(x_i)\}_{i\in I}$ such that $x_i\in \partial E$ and the dilated balls $\{B_{3\ell}(x_i)\}_{i\in I}$ cover the region $\{x\in E\Delta F: d_F(x) \leq \ell\}$. By the density estimates in \cref{density lemma for almost minimizers}, it follows that \begin{align*}
        |\{x\in E\Delta F: d_F(x) \leq \ell\}| &\leq \sum_{i\in I} |B_{3\ell}(x_i)| \\
        &\leq 3^N C \ell \sum_{i\in I} P_\phi(E; B_\ell(x_i))\\
        &\leq 3^N C \ell P_\phi(E)
    \end{align*}
    where $C = C(N,L_\phi,L_\psi)$.
\end{proof}

\begin{lemma}[H\"older continuity in time, \cref{lem:Holder continuity in time}]
    \label{approx Holder continuity in time}
    Let $h\leq 1$ and let $\{\Eh(t)\}_{t\geq0}$ be an approximate flat flow. Then for all $0\leq s\leq t <\infty$ and a constant $C=C(N,L_\phi,L_\psi)>0$,
    \begin{equation}
        |\Eh(s)\Delta \Eh(t)| \leq CP_\phi(E_0)\max\{h,|t-s|\}^{1/2}.
    \end{equation}
\end{lemma}
\begin{proof}
    We may assume $t-s \geq h$, for otherwise, we can replace $t$ with $s+h$. Applying \cref{lem:L1 estimate} for any $0<\ell \leq c\sqrt{h}$, we may bound \begin{align*}
        |\Eh(t)\Delta \Eh(s)| &\leq \sum_{j=\floor{s/h}}^{\floor{t/h}-1} |\Eh_{(j+1)h}\Delta \Eh_{jh}| \\
        &\leq C \sum_{j=\floor{s/h}}^{\floor{t/h}-1} \ps{\ell P_\phi (\Eh_{(j+1)h}) + \inv{\ell} \cD^\psi(\Eh_{(j+1)h}, \Eh_{jh})}.
    \end{align*}
    Then \eqref{iterated dissipation} simplifies the bound to \begin{equation}
        |\Eh(t) \Delta \Eh(s)| \leq C \ps{\ell \frac{|t-s|}{h} P_\phi(E_0) + \frac{h}{\ell} P_\phi(E_0)}.
    \end{equation}
    In particular, we may set $\ell := c \frac{h}{|t-s|^{1/2}} \leq c\sqrt{h}$ to balance the righthand terms: \[ |\Eh(t) \Delta \Eh(s)| \leq C |t-s|^{1/2} P_\phi(E_0). \qedhere \]
\end{proof}

\begin{lemma}[$L^2$ estimate, \cref{lem:standard estimate flat flow}(iii)]
    \label{L2 estimate}
    Let $F\sb\R^N$ be a bounded set of finite perimeter and let $E$ be a minimizer of $\cF_h(\cdot, F)$. Then there exist a constant $C=C(N,L_\phi,L_\psi)>0$ such that 
        \begin{equation}
            \label{L2 estimate eq}
            \int_{\partial^* E} (d^\psi_F)^2(x)d\cH^{N-1} \leq C \cD^\psi(E,F). 
        \end{equation}
\end{lemma}
\begin{proof}
    Let $\aa := 2L_\psi$. For any $k\in\Z$ define $A_k := \{x\in\R^N: \aa^k < d^\psi_F(x) \leq \aa^{k+1} \}$. Note that $\partial E$ is covered by the collection of $A_k$ over all $k$ such that $\aa^k\leq c\sqrt{h}$, where $c$ is the constant from \cref{Linfty estimate}. 
    
    For any $x\in \partial E \cap A_k$, we note by the triangle inequality that $\aa^{k-1} < d^\psi_F(y) \leq \aa^{k+2}$ for all $y\in B_{\aa^{k-1}}(x)$, and in particular $B_{\aa^{k-1}}(x)$ is contained in either $F$ or $F^c$. Thus by the density estimates in \cref{density lemma for almost minimizers}, we obtain the bound \begin{equation*}
        \int_{(E\Delta F)\cap B_{\aa^{k-1}}(x)} d^\psi_F(y)dy \geq \aa^{k-1} \min\set{|E\cap B_{\aa^{k-1}}(x)|, |B_{\aa^{k-1}(x)}\setminus E|} \gtrsim_{N,L_\phi,L_\psi} \aa^{k-1}\aa^{N(k-1)} = \aa^{Nk - N + k}
    \end{equation*}
    and similarly
    \begin{equation*}
        \int_{\partial E\cap B_{\aa^{k-1}}(x)} (d^\psi_F)^2 d\cH^{N-1} \lesssim_{N,L_\phi,L_\psi} (\aa^{k+2})^2 \aa^{(N-1)(k-1)} = \aa^{Nk -N + k+5}.
    \end{equation*}
    Altogether we obtain
    \[ 
        \int_{\partial E\cap B_{\aa^{k-1}}(x)} (d^\psi_F)^2 d\cH^{N-1} \lesssim_{N,L_\phi,L_\psi} \int_{(E\Delta F)\cap B_{\aa^{k-1}}(x)} d^\psi_F(y)dy. 
    \]
    By Besicovitch's covering lemma, $\partial E \cap A_k$ can be covered by finitely many disjoint subcollections of $\set{B_{\aa^{k-1}}(x): x\in \partial E \cap A_k}$, where the number of subcollections is bounded by a dimensional constant. It follows that \[ \int_{\partial E\cap A_k} (d^\psi_F)^2 d\cH^{N-1} \lesssim_{N,L_\phi,L_\psi} \int_{(E\Delta F)\cap [A_{k-1}\cup A_k\cup A_{k+1}]} d^\psi_F(y)dy. \]
    Summing over all $k$ such that $\aa^{k}\leq c\sqrt{h}$ yields the desired result.
\end{proof}

\begin{proposition}
\label{prop:l2}
Let $\{\Eh(t)\}_{t\geq0}$ be a flat $(\phi,\psi)$-flow. Then 
\[ \int_0^\infty \int_{\partial^* \Eh(t)}\vh(t,x)^2 d\cH^{N-1} dt \leq C P_\phi(E_0) \]
where $C = C(N,L_\phi,L_\psi)$.
\end{proposition}
\begin{proof}
For any $T>0$, we have by \cref{L2 estimate} and dissipation that 
    \begin{align*}
        \int_0^T \int_{\partial^* \Eh(t)}\vh(t,x)^2 d\cH^{N-1} dt &\leq \inv{h} \sum_{i=0}^{\floor{T/h}-1} \int_{\partial^*{\Eh_i}} (d^\psi_{\Eh_{i+1}})^2 d\cH^{N-1}\\
        &\leq \frac{C}{h} \sum_{i=0}^{\floor{T/h}-1} \cD^\psi(\Eh_{i+1}, \Eh_i)\\
        &\leq CP_\phi(E_0).
    \end{align*}
The result follows by sending $T\to\infty$.
\end{proof}

\section{Reflection property in two dimensions}
\label{ap:ref}

Let us show that any two-dimensional cross-section of a finite root system $\mathcal{P}$ given in \eqref{eq:root system} is $Q_{2m}$ for $m \in \N$, up to a rotation, where
\begin{align}
\label{eqn:pm}
    Q_{j} := \left\{  \left(\cos\left(\frac{2\pi i}{j}\right), \sin\left(\frac{2\pi i}{j}\right) \right) : 1 \leq i \leq j \right\} \hbox{ for } j \in \N.
\end{align}

\begin{proposition}
    
\label{prop:ppi}
Let $\mathcal{P}$ be a finite root system in $\mathbb{R}^N$. For any two dimensional hyperplane $\Pi$ in $\mathbb{R}^N$, if $\mathcal{P} \cap \Pi$ is nonempty, then there exists $m \in \N$ such that
\begin{align}
\label{eqn:ppi}
\mathcal{P} \cap \Pi = Q_{2m}
\end{align}
up to a rotation. Here $Q_{2m}$ is given in \eqref{eqn:pm}.
\end{proposition}

\begin{proof}
First, if $\mathcal{P} \cap \Pi = \{ \pm p \}$ for some $p\in \S^{n-1}$, then $\mathcal{P} \cap \Pi = Q_{2}$ up to a rotation.

\medskip

Suppose that $\mathcal{P} \cap \Pi$ contains at least two linearly independent vectors. Choose $p_1$ and $p_2$ in $\mathcal{P} \cap \Pi$ such that the angle $\theta_{21}>0$ between the two vectors is smallest among all pairs of two linearly independent vectors in $\mathcal{P} \cap \Pi$. As $\mathcal{P}$ is a root system, we can find a sequence of vectors $p_i$ in $\mathcal{P} \cap \Pi$ given by
\begin{align}
p_i := 2(p_{i-1} \cdot p_{i-2}) p_{i-1} - p_{i-2} \hbox{ for } i \geq 3.
\end{align}
\medskip

As $\mathcal{P}$ is finite, there exists $i^* \geq 2$ such that $p_1 = p_{i^*+1}$ and thus 
\begin{align}
i^* \theta_{21} = 2k \pi
\end{align}
for some $k \in \N$. Here, $\theta_{21}>0$ denotes the angle between $p_1$ and $p_2$. As $\theta_{21}$ is smaller than or equal to the angle between $p_i$ and $p_j$ for any $i,j \in \N$, we obtain that $k = 1$, $i^* = 2m$ for some $m \in \mathbb{N}$ and $\mathcal{P}\cap\Pi = \{p_i\}_{i \in \{1,2, \cdots, i^*\}}$.
\end{proof}

\begin{proposition}
\label{prop:b}
Under the same setting as in \cref{prop:reflection lipschitz}, we additionally assume $N=2$. Then $E$ is star-shaped with respect to the origin.
\end{proposition}

\begin{proof}
    Recall from \cref{prop:ppi} that $\mathcal{P} = Q_{2m}$ for some $m \in \mathbb{N}$ and $m \geq 3$. Applying  \cite[Theorem 2.7]{kim2021volume} to the assumption $\rho < c|E|^{1/N}$, there exists $r>0$ such that
    \begin{align}
    \label{eq:sett0}
        B_{\sigma_1^{-1} \sigma_2 (\rho + 2r)}(0) \subset E
    \end{align}
    where $\sigma_1 = \cos \frac{\pi}{m}$ and $\sigma_2 = 1/\cos \frac{\pi}{2m}$.
    
    Choose $x_0 \in E \setminus B_{\sigma_1^{-1} \sigma_2 (\rho + 2r)}(0)$. There exist $p_1, p_2 \in \mathcal{P}$ such that 
    \begin{align}
    \label{eq:sett}
        x_0 = a_1 p_1 + a_2 p_2, a_1 \geq 0,   a_2 \geq 0, \hbox{ and } p_1 \cdot p_2 = \cos \frac{\pi}{m}.
    \end{align}
    Applying the reflection property $(**)_{B_\rho(0),\cP}$ iteratively, as in \cite[Lemma 2.6]{kim2021volume}, we have that
    \begin{align}
        \mathcal{I}_1:= (x_0 - \cone_\infty(\{p_1, p_2\})) \cap \{x: x \cdot p_1 \geq \rho \hbox { and } x \cdot p_2 \geq \rho\} \subset E
    \end{align}
    where $\cone_\infty(\{p_1,p_2\}) := \set{c_1p_1 + c_2p_2: c_1,c_2\geq0}$. Thanks to \eqref{eq:sett}, $B_{\sigma_1^{-1} \sigma_2 (\rho + 2r)}(0)$ contains the region \begin{align}
        \mathcal{I}_2 := \cone_\infty(\{p_1, p_2\}) \cap \{x: x \cdot p_1 < \rho \hbox{ or } x \cdot p_2 < \rho\}.
    \end{align}
    Because the line segment joining $x_0$ to the origin is covered by $\mathcal{I}_1$ and $\mathcal{I}_2$, $E$ is star-shaped.
\end{proof}

\bibliographystyle{alpha}
\bibliography{main}

\end{document}